\documentclass[a4paper]{scrartcl}
\usepackage[T1]{fontenc}
\usepackage[english]{babel}
\usepackage[latin1]{inputenc}
\usepackage{amsmath}
\usepackage{amsthm,amssymb}
\usepackage{csquotes}
\usepackage{amssymb}
\usepackage{graphicx}
\usepackage{color, pgf}
\usepackage{hyperref}
\usepackage{authblk}
\newtheorem{thm}{Theorem}[subsection]

\newtheorem{cor}[thm]{Corollary}

\newtheorem{prop}[thm]{Proposition}
\newtheorem{theorem}{Theorem}[section]
\newtheorem{definition}[theorem]{Definition}
\newtheorem{lemma}[theorem]{Lemma}
\newtheorem{corollary}[theorem]{Corollary}
\newtheorem{remark}[theorem]{Remark}
\newtheorem{proposition}[theorem]{Proposition}
\begin{document}
\title{Well-posedness of renormalized solutions for a stochastic $p$-Laplace equation with $L^1$-initial data}
\author{Niklas Sapountzoglou, Aleksandra Zimmermann}
\affil{Faculty of Mathematics, University of Duisburg-Essen, Thea-Leymann-Strasse 9, 45127 Essen, Germany\\
E-mail: niklas.sapountzoglou@stud.uni-due.de, aleksandra.zimmermann@uni-due.de}
\maketitle
\begin{abstract}
We consider a $p$-Laplace evolution problem with stochastic forcing on a bounded domain $D\subset\mathbb{R}^d$ with homogeneous Dirichlet boundary conditions for $1<p<\infty$. The additive noise term is given by a stochastic integral in the sense of It\^{o}. The technical difficulties arise from the merely integrable random initial data $u_0$ under consideration. Due to the poor regularity of the initial data, estimates in $W^{1,p}_0(D)$ are available with respect to truncations of the solution only and therefore well-posedness results have to be formulated in the sense of generalized solutions. We extend the notion of renormalized solution for this type of SPDEs, show well-posedness in this setting and study the Markov properties of solutions.
\end{abstract}
\section{Introduction}
\subsection{Motivation of the study}
We are interested in the study of well-posedness for a $p$-Laplace evolution problem with stochastic forcing on a bounded domain $D\subset\mathbb{R}^d$ with homogeneous Dirichlet boundary conditions for $1<p<\infty$. For $p=2$, we are in the case of the classical Laplace operator, for arbitrary $1<p<\infty$, $u\mapsto-\operatorname{div}\,(|\nabla u|^{p-2}\nabla u)$ is a monotonone operator on the Sobolev space $W^{1,p}_0(D)$ that is singular for $p<2$ and degenerate for $p>2$. Evolution equations of $p$-Laplace type may appear as continuity equations in the study of gases flowing in pipes of uniform cross sectional areas and in models of filtration of an incompressible fluid through a porous medium (see \cite{BB}, \cite{DTH}): In the case of a turbulent regime, a nonlinear version of the Darcy law of $p$-power law type for $1<p<2$ is more appropriate (see \cite{DTH}). Turbulence is often associated with the presence of randomness (see \cite{BFH} and the references therein). Adding random influences to the model, we also take uncertainties and multiscale interactions into account. Randomness may be introduced as random external force by adding an It\^{o} integral on the right-hand side of the equation and by considering random initial values. Consequently, the equation becomes a stochastic partial differential equation (SPDE) and the solution is then a stochastic process.\\
For square-integrable initial data $u_0$, the stochastic $p$-Laplace evolution problem is well-posed (see, e.g. \cite{EP}, \cite{WLMR}). In this contribution, we focus on more general, merely integrable random initial data. There has been an extensive study of the corresponding deterministic problem and its generalizations (see, e.g., \cite{DBFMHR}, \cite{DB}, \cite{DB2}) and from these results it is well known that the deterministic $p$-Laplace evolution problem is not well-posed in the variational setting for initial data in $L^1$ and $1<p<d$, were $d\in\mathbb{N}$ is the space dimension. For this reason, the problem is formulated in the framework of renormalized solutions. The notion of renormalization summarizes different strategies to get rid of infinities (see \cite{DeLM}) that may appear in physical models. It has been introduced to partial differential equations by Di Perna and Lions in the study of Boltzmann equation (see \cite{DL}) and then extended to many elliptic and parabolic problems (see, e.g., \cite{PBLBTGRGMPJLV}, \cite{DB2}, \cite{BR} and the references therein). Properties of renormalized solutions for the continuity equation of viscous compressible fluids have been studied in \cite{EF}. The basic idea of the classical renormalized formulation for PDE is to use an appropriate class of nonlinear functions of the solution as test functions in the equation. For SPDEs, this concept has been applied for stochastic transport equations in \cite{AF}, \cite{CO} and for the Boltzmann equation with stochastic kinetic transport in \cite{PS}. For many physically relevant singular SPDEs, a slightly different notion of renormalization has recently been developed (see \cite{GIP2015}, \cite{Hai2014} and the references therein). For these cases, renormalized solutions may be obtained as limits of classical solutions to regularized problems with addition of diverging correction terms. These counterterms arise from a renormalization group which is defined in terms of an associated regularity structure.\\
In this contribution, it is our aim to extend the notion of renormalized solutions in the sense of \cite{DB2} for the stochastic $p$-Laplace evolution problem with random initial data in $L^1$ and to show well-posedness in this framework. 

For a quasilinear, degenerate hyperbolic-parabolic SPDE with $L^1$ random initial data, the well-posedness and regularity of kinetic solutions has been studied in \cite{BGMH}, but, to the best of our knowledge, these results do not apply in our situation. 

\subsection{Statement of the problem}
Let $(\Omega, \mathcal{F}, P,(\mathcal{F}_t)_{t \in [0,T]}, (\beta_t)_{t\in [0,T]})$ be a stochastic basis with a complete, countably generated probability space $(\Omega, \mathcal{F}, P)$, a filtration $(\mathcal{F}_t)_{t \in [0,T]}\subset \mathcal{F}$ satisfying the usual assumptions and a real valued, $\mathcal{F}_t$-Brownian motion $(\beta_t)_{t\in [0,T]}$. Let  $D \subset \mathbb{R}^d$ be a bounded Lipschitz domain, $T>0$, $Q_T=(0,T) \times D$ and $p>1$. Furthermore, let $u_0: \Omega \to L^1(D)$ be $\mathcal{F}_0$-measurable and $\Phi \in L^2(\Omega; \mathcal{C}([0,T];L^2(D)))$ be predictable.\\
We are interested in well-posedness to the following stochastic $p$-Laplace evolution problem
\begin{align}\label{1}
du - \textnormal{div}\,(|\nabla u|^{p-2} \nabla u)\,dt &= \Phi ~d\beta ~~~&\textnormal{in}~ \Omega \times Q_T, \notag\\
u&=0 ~~~&\textnormal{on}~ \Omega \times (0,T) \times \partial D, \\
u(0, \cdot)&=u_0 ~~~&\in ~L^1(\Omega \times D). \notag
\end{align}
Due to the poor regularity of the initial data $u_0$, a-priori estimates on $\nabla u$ are not available and therefore the well-posedness result has to be formulated in the sense of a generalized solution, more precisely in the framework of renormalized solutions.
To show this we first show in Section 2 that there exists a strong solution to \eqref{1} in the case where the initial value $u_0$ is an element of $L^2(\Omega \times D)$. After that, we establish a comparison principle that shows that a sequence of strong solutions is a Cauchy sequence in $L^1(\Omega;\mathcal{C}([0,T]; L^1(D)))$ whenever the sequence of initial values is a Cauchy sequence in $L^1(\Omega \times D)$. In Section 4 we prove a version of the It\^{o} formula which makes it possible to define renormalized solutions to equation \eqref{1}. Section 5 contains the definition of renormalized solutions to \eqref{1}, in Section 6 we show the existence of such a solution and Section 7 contains the uniqueness result, which is based on an $L^1$-contraction principle. Finally, in Section 8 we study the Markov properties of such a solution.
\section{Strong solutions}
\begin{theorem}\label{Theorem 2.1}
Let the conditions in the introduction be satisfied. Furthermore, let $u_0 \in L^2(\Omega \times D)$ be $\mathcal{F}_0$-measurable. Then there exists a unique strong solution to (\ref{1}), i.e., an $\mathcal{F}_t$-adapted stochastic process $u:\Omega\times [0,T]\to L^2(D)$ such that $u \in L^p(\Omega; L^p(0,T; W_0^{1,p}(D))) \cap L^2(\Omega; \mathcal{C}([0,T];L^2(D)))$, $u(0, \cdot)=u_0$ in $L^2(\Omega \times D)$ and
\begin{align*}
u(t) - u_0 - \int_0^t \textnormal{div}\,(|\nabla u|^{p-2} \nabla u)\, ds = \int_0^t \Phi \,d\beta
\end{align*}
in $W^{-1,p'}(D) + L^2(D)$ for all $t \in [0,T]$ and a.s. in $\Omega$.
\end{theorem}
\begin{remark}\label{Remark 2.2}
Since we know from all terms except the term $\int_0^t \textnormal{div}\,(|\nabla u|^{p-2} \nabla u) ~ds $ that these terms are elements of $L^2(D)$ for all $t \in [0,T]$ and a.s. in $\Omega$ it follows that $\int_0^t \textnormal{div} (|\nabla u|^{p-2} \nabla u) ~ds  \in L^2(D)$ for all $t \in [0,T]$ and a.s. in $\Omega$. Therefore this equation is an equation in $L^2(D)$.
\end{remark}
\begin{proof}
The existence result is a consequence of \cite{NVKBLR}, Chapter II, Theorem 2.1 and Corollary 2.1. We only have to check the assumptions of this theorem. Following the notations therein, we set $V= W_0^{1,p}(D) \cap L^2(D)$ in the case $1<p<2$ and $V=W_0^{1,p}(D)$ in the case $p \geq 2$, $H=L^2(D)$, $E=\mathbb{R}$, $A: V \to V^*$, $A(u) = - \textnormal{div}\,(|\nabla u|^{p-2} \nabla u)$, $B=\Phi$, $ f(t,\omega) = 2 + \Vert B(t,\omega) \Vert_2^2$ for almost each $(t,\omega) \in (0,T) \times \Omega$ and $z=0$. Then we have $\mathcal{L}_Q(E;H)= \mathcal{L}_2(\mathbb{R}, L^2(D)) = L^2(D)$.\\
We remark that $A$ does not depend on $(t,\omega) \in [0,T] \times \Omega$ and that $B$ does not depend on $u \in V$. Obviously, conditions (A1), (A2) and (A5) in \cite{NVKBLR} are satisfied. Moreover, in the case $p \geq 2$ the validity of conditions (A3) and (A4) is well known in the theory of monotone operators. Therefore we only consider the case $1<p<2$.\\
In this case we check condition (A3). Using the norms
\begin{align*}
\Vert v \Vert_V &:= \bigg( \Vert v \Vert_{W_0^{1,p}(D)}^p + \Vert v \Vert_2^p \bigg)^{\frac{1}{p}}, \\
\Vert v \Vert_{W_0^{1,p}(D)}&:= \Vert \nabla v \Vert_{L^p(D)^d}
\end{align*}
we have
\begin{align*}
|B|_Q^2 + 2 \Vert v \Vert_V^p &= \Vert B \Vert_2^2 + 2 \Vert v \Vert_V^p \\
&= f-2 + 2 \Vert v \Vert_V^p\\
&= f - 2 + 2 \Vert v \Vert_{W_0^{1,p}(D)}^p + 2 \Vert v \Vert_2^p \\
&= f-2 + 2 \Vert v \Vert_2^p +  2\langle Av,v \rangle_{V^*,V}\\
&\leq f + \Vert v \Vert_2^2  + 2\langle Av,v \rangle_{V^*,V}
\end{align*}
for all $v \in V$ since $x^p \leq 1 + x^2$ for all $x \geq 0$. This proves condition (A3) for $\alpha=K=2$. \\
Now we check condition (A4). We estimate
\begin{align*}
\Vert A(u) \Vert_{V^*} \leq \Vert A(u) \Vert_{W^{-1,p'}(D)} \leq \Vert \nabla u \Vert_{L^p(D)^d}^{p-1} \leq \Vert u \Vert_V^{p-1}.
\end{align*}
The uniqueness is a consequence of \cite{NVKBLR}, Chapter II, Theorem 3.2, which applies under the same assumptions.
\end{proof}
\section{Comparison principle}
\begin{theorem}\label{Theorem 3.1}
Let $u_0, v_0 \in L^2(\Omega \times D)$ and $u$ and $v$ strong solutions to the problem \eqref{1} with initial value $u_0$ and $v_0$, respectively. Then
\begin{align*}
\sup\limits_{t \in [0,T]} \int_D |u(t) - v(t)| \,dx \leq \int_D |u_0 - v_0| \,dx
\end{align*}
a.s. in $\Omega$.
\end{theorem}
\begin{proof}
We subtract the equations for $u$ and $v$ and we get
\begin{align*}
u(t) - v(t) - (u_0 - v_0) - \int_0^t \textnormal{div} (|\nabla u|^{p-2} \nabla u - |\nabla v|^{p-2} \nabla v) \,ds = 0
\end{align*}
for all $t \in [0,T]$ and a.s. in $\Omega$.\\
Using the It\^{o} formula with an approximation of the absolute value and tending to the limit yields (see, e.g., Proposition 5 in \cite{VZ18})
\begin{align*}
\int_D |u(t) - v(t)| \,dx - \int_D |u_0 - v_0| \,dx \leq 0
\end{align*}
for all $t \in [0,T]$ and a.s. in $\Omega$.
\end{proof}
\section{It\^{o} formula and renormalization}
In order to find an appropriate notion of renormalized solutions to \eqref{1}, we prove an It\^{o} formula in the $L^1$-framework. We remark that the combined It\^{o} chain and product rule from \cite{BFH}, Appendix A4 does not apply to our situation for two reasons. Firstly, we take the bouded domain $D\subset\mathbb{R}^d$ into account in our regularizing procedure by adding a cutoff function (see Appendix, Subsection \ref{proof}). Secondly, the spacial regularities are different in our case.\\
For two Banach spaces $X$, $Y$, let $L(X;Y)$ denote the Banach space of bounded, linear operators from $X$ to $Y$ and $L(X)$ denote the space of bounded linear operators from $X$ to $X$ respectively.\\
For the sake of completeness, we recall the following regularization procedure: 
\begin{lemma}\label{190212_lem01}
Let $D\subset\mathbb{R}^d$ be bounded domain with Lipschitz boundary, $1\leq p<\infty$. There exists a sequence of operators
\[\Pi_n:W^{-1,p'}(D)+L^1(D)\rightarrow W^{1,p}_0(D)\cap L^{\infty}(D),~n\in\mathbb{N}\]
such that 
\begin{itemize}
\item[$i.)$] $\Pi_n(v)\in W^{1,p}_0(D)\cap C^{\infty}(\overline{D})$ for all $v\in W^{-1,p'}(D)+L^1(D)$ and all $n\in\mathbb{N}$ 
\item[$ii.)$] For any $n\in\mathbb{N}$ and any Banach space 
\[F\in \{W^{1,p}_0(D), L^2(D), L^1(D), W^{-1,p'}(D), W^{-1,p'}(D)+L^1(D)\}.\]
$\Pi_n : F \to F$ is a bounded linear operator such that 
$\lim_{n \to \infty}{\Pi_n}_{|F}=I_F$ pointwise in $F$, where $I_F$ is the identity on $F$.
\end{itemize}
\end{lemma}
\begin{proof}
See Appendix, Subsection \ref{proof}.
\end{proof}

\begin{proposition}\label{Proposition 4.2}
Let $G \in L^{p'}(\Omega \times Q_T)^d$, $g\in L^2(\Omega\times Q_T)$, $f\in L^1(\Omega\times Q_T)$ be progressively measurable, $u_0 \in L^1(\Omega \times D)$ be $\mathcal{F}_0$-measurable and $u \in L^1(\Omega; \mathcal{C}([0,T];L^1(D))) \cap L^p(\Omega; L^p(0,T;W_0^{1,p}(D)))$ satisfy the equality
\begin{align}\label{2}
u(t) - u_0 +\int_0^t (-\textnormal{div}\,G +f) \,ds = \int_0^t g \,d\beta
\end{align}
in $L^2(D)$ for all $t \in [0,T]$ and a.s. in $\Omega$.\\
Then for all $\psi \in C^{\infty}([0,T] \times \overline{D})$ and all $S\in W^{2,\infty}(\mathbb{R})$ with $S''$ piecewise continuous 
such that $S'(0)=0$ or $\psi(t,x) =0$ for all $(t,x) \in [0,T]\times \partial D$ we have
\begin{align}\label{Itoformulat}
&(S(u(t)),\psi(t))_2 - (S(u_0),\psi(0))_2+ \int_0^t \langle -\operatorname{div}\,G+f,S'(u)\psi\rangle\,ds\nonumber\\ 
&=\int_0^t (S'(u)g,\psi)_2\,d\beta + \int_0^t (S(u),\psi_t)_2 \,ds + \frac{1}{2} \int_0^t \int_D  S''(u)g^2 \psi  \,dx\,ds
\end{align}
for all $t \in [0,T]$ and a.s. in $\Omega$, where
\begin{align*}
&\langle -\operatorname{div}\,G+f,S'(u)\psi\rangle=\langle -\operatorname{div}\,G+f,S'(u)\psi\rangle_{W^{-1,p'}(D)+L^1(D), W^{1,p}_0(D)\cap L^{\infty}(D)}\\
&=\int_D(G\cdot\nabla[S'(u)\psi]+fS'(u)\psi)\, dx
\end{align*}
a.s. in $\Omega\times (0,T)$.
In particular, for $\psi \in \mathcal{C}^{\infty}(\overline{D})$ not depending on $t\in [0,T]$ we get
\begin{align}\label{Itoformula}
&\int_D (S(u(t)) - S(u_0)) \psi \,dx + \int_0^t \int_D G \cdot\nabla[S'(u)\psi] \,dx\,ds + \int_0^t \int_D fS'(u)\psi \,dx\,ds\nonumber\\
= &\int_0^t \int_D S'(u) \psi g \,dx\,d\beta + \frac{1}{2} \int_0^t \int_D  S''(u) \psi g^2 \,dx\,ds
\end{align}
for all $t \in [0,T]$ and a.s. in $\Omega$.
\end{proposition}
\begin{proof}
Let us assume $S\in \mathcal{C}^2(\mathbb{R})$ such that $S'$, $S''$ is bounded, the general result then follows by an approximation argument (see Corollary \ref{cor1} in the Appendix).\\
We choose the regularizing sequence $(\Pi_n)$ according to Lemma \ref{190212_lem01} and set $u_n:= \Pi_n(u)$, $u_0^{n}:= \Pi_n(u_0)$, $G_n:= \Pi_n(-\operatorname{div}\,G)$, $f_n:=\Pi_n(f)$ and $g_n:=\Pi_n(g)$. We apply the operator $\Pi_n$ to both sides of \eqref{2}. Since $\Pi_n \in L(W^{-1,p'}(D)+L^1(D); W_0^{1,p}(D)\cap L^{\infty}(D))$, we may conclude
\begin{align*}
u_n(t) - u_0^n +\int_0^t G_n+f_n \,ds= \int_0^t g_n \,d\beta
\end{align*}
in $D$, for all $t \in [0,T]$ and a.s. in $\Omega$. For $x \in D$ fixed, we apply the classic It\^{o} formula for $h(t,r):=S(r)\psi(t,x)$ with respect to the time variable $t$. Integration over $D$ afterwards and Fubini Theorem yield
\begin{align*}
I_1+I_2+I_3=I_4+I_5+\frac{1}{2}I_6,
\end{align*}
where
\begin{align*}
I_1=\int_D S(u_n(t))\psi(t) - S(u_0^n) \psi(0) \,dx,
\end{align*}
\begin{align*}
I_2=\int_0^t \langle G_n, S'(u_n) \psi \rangle_{W^{-1,p'}(D),W_0^{1,p}(D)}\,ds,
\end{align*}
\begin{align*}
I_3=\int_0^t\int_D f_nS'(u_n) \psi\,dx \,ds,
\end{align*}
\begin{align*}
I_4=\int_0^t \int_D S'(u_n) \psi g_n \,dx\,d\beta
\end{align*}
\begin{align*}
I_5=\int_0^t \int_D S(u_n) \psi_t \,dx\,ds
\end{align*}
\begin{align*} 
I_6=\int_0^t \int_D  S''(u_n) \psi g_n^2 \,dx\,ds
\end{align*} 
for all $t \in [0,T]$ and a.s. in $\Omega$. Now, we want to pass to the limit with $n\to \infty$ in $I_1-I_6$. Since $u_0^n\rightarrow u_0$ and $u_n(t)\rightarrow u(t)$ in $L^1(D)$ a.s. in $\Omega$ for any $t\in [0,T]$,
\begin{align}\label{190314_01}
\lim_{n\rightarrow\infty} I_1=\int_D S(u(t))\psi(t) - S(u_0) \psi(0) \,dx.
\end{align}
For any $s\in (0,t)$ and a.s. in $\Omega$, $G_n(\omega,s)\rightarrow -\operatorname{div}\, G(\omega,s)$ in $W^{-1,p'}(D)$ for $n\rightarrow\infty$. Moreover,
\begin{align*}
\Vert G_n(\omega,s)\Vert_{W^{-1,p'}(D)}&\leq \Vert \Pi_n\Vert_{L(W^{-1,p'}(D))}\Vert -\operatorname{div}\, G(\omega,s)\Vert_{W^{-1,p'}(D)}\\
&\leq C_U\Vert -\operatorname{div}\, G(\omega,s)\Vert_{W^{-1,p'}(D)},
\end{align*}
where $C_U\geq 0$ is a generic constant not depending on $n\in\mathbb{N}$ from the Uniform Boundedness Principle. Since $G(\omega,\cdot)\in L^{p'}(Q_T)^d$ for a.e. $\omega\in\Omega$, it follows that $-\operatorname{div}\, G(\omega,\cdot)\in L^{p'}(0,T;W^{-1,p'}(D))$ and from Lebesgue's dominated convergence theorem it follows that
\begin{align*}
\lim_{n\rightarrow\infty} G_n=-\operatorname{div}\, G
\end{align*}
in $L^{p'}(0,t;W^{-1,p'}(D))$ for every $t\in (0,T)$, a.s. in $\Omega$. For every $s\in (0,t)$ and a.e. $\omega\in\Omega$, from the chain rule for Sobolev functions we get
\begin{align}\label{190315_01}
\nabla[S'(u_n(\omega,s))\psi(s)]=S''(u_n(\omega,s))\nabla u_n(\omega,s)\psi(s)+S'(u_n(\omega,s))\nabla \psi(s).
\end{align}
For any $s\in [0,t]$ and almost every $\omega\in\Omega$, $u_n(\omega,s)\rightarrow u(\omega,s)$ in $W^{1,p}_0(D)$ for $n\rightarrow\infty$, passing to a (not relabeled) subsequence if necessary (that may depend on $(\omega,s)$), the right-hand side of \eqref{190315_01} converges to $S''(u(\omega,s))\nabla u(\omega,s)\psi(s)+S'(u(\omega,s))\nabla\psi(s)$ for $n\rightarrow\infty$ a.e. in $D$ and there exists $\zeta\in L^p(D)$, that may depend on $(\omega,s)$, such that
\begin{align*}
|u_n(\omega,s)|+|\nabla u_n(\omega,s)|\leq \zeta(\omega,s)
\end{align*}
for all $n\in\mathbb{N}$, a.s. in $D$. Consequently, $S'(u_n(\omega,s))\psi(s)\rightarrow S'(u(\omega,s))\psi(s)$ for $n\rightarrow\infty$ in $W^{1,p}_0(D)$ and this convergence holds for the whole sequence. From the boundedness of $S'$, $S''$, $\psi$ and $\nabla \psi$ it follows that there exist constants $C, \tilde{C}\geq 0$ not depending on the parameters $n,\omega,s$ such that
\begin{align*}
\Vert S'(u_n(\omega,s))\psi(s)\Vert_{W^{1,p}_0(D)}=C\Vert \Pi_n\Vert_{L(W^{1,p}_0(D))}\Vert u(\omega,s)\Vert_{W^{1,p}_0(D)}+\tilde{C}
\end{align*}
and $\Vert \Pi_n\Vert_{L(W^{1,p}_0(D))}\leq C_U$ for all $n\in\mathbb{N}$ thanks to the Uniform Boundedness Principle. For these reasons, from Lebesgue's dominated convergence theorem it follows that 
\begin{align}\label{190315_03}
\lim_{n\rightarrow\infty}S'(u_n)\psi=S'(u)\psi 
\end{align}
in $L^p(0,t;W^{1,p}_0(D))$ a.s. in $\Omega$ and therefore
\begin{align}\label{190315_02}
\lim_{n\rightarrow\infty} I_2=\int_0^t \langle -\operatorname{div}\, G, S'(u) \psi \rangle_{W^{-1,p'}(D),W_0^{1,p}(D)}\,ds.
\end{align}
a.s. in $\Omega$. For any $s\in (0,t)$ and a.e. $\omega\in \Omega$, $f_n(\omega,s)\rightarrow f(\omega,s)$ in $L^1(D)$. Moreover,
\begin{align*}
\Vert f_n(\omega,s)\Vert_{L^1(D)}\leq C_U\Vert f(\omega,s)\Vert_{L^1(D)}
\end{align*}
for all $n\in\mathbb{N}$, for all $s\in (0,t)$ and a.s. in $\Omega$. Therefore,
from Lebesgue's dominated convergence theorem it follows that $f_n\rightarrow f$ in $L^1((0,t)\times D)$ a.s. in $\Omega$ for $n\rightarrow\infty$. On the other hand, since $S'(u_n)\psi$ is bounded with respect to $n\in\mathbb{N}$ in $L^{\infty}(Q_T)$ and from the convergence \eqref{190315_03} in $L^p(0,t;W^{1,p}_0(D))$ it follows that $S'(u_n)\psi\stackrel{\ast}{\rightharpoonup}S'(u)\psi$ in $L^{\infty}(Q_t)$ a.s. in $\Omega$, therefore
\begin{align}\label{190315_04}
\lim_{n\rightarrow\infty}I_3=\int_0^t\int_D fS'(u)\psi\,dx\,ds
\end{align}
a.s. in $\Omega$. Using It\^{o} isometry we get that
\begin{align}\label{190318_01}
&\mathbb{E}\left|\int_0^t\int_D S'(u_n)\psi g_n-S'(u)\psi g\,dx\,d\beta\right|^2\nonumber\\
&=\mathbb{E}\int_0^t\int_D |S'(u_n)\psi g_n-S'(u)\psi g|^2\,dx\,ds\nonumber\\
&\leq 2\Vert \psi\Vert_{\infty}^2\mathbb{E}\left[\int_0^t\int_D|S'(u_n)(g_n-g)|^2\,dx\,ds+\int_0^t\int_D|(S'(u_n)-S'(u))g|^2\,dx\,ds\right]\nonumber\\
&\leq 2\Vert \psi\Vert_{\infty}^2\mathbb{E}\left[\Vert S'\Vert_{\infty}^2\int_0^t\Vert g_n-g\Vert^2_{L^2(D)}\,ds+\int_0^t\int_D|(S'(u_n)-S'(u))g|^2\,dx\,ds\right].
\end{align}
Since $g_n(\omega,s)\rightarrow g(\omega,s)$ for $n\rightarrow\infty$ a.s. in $\Omega\times (0,T)$ and
\begin{align*}
\Vert g_n(\omega,s)-g(\omega,s)\Vert^2\leq 2\Vert g(\omega,s)\Vert^2_{L^2(D)}(C_U+1),
\end{align*}
from Lebesgue's dominated convergence theorem it follows that
\begin{align}\label{190318_03}
\lim_{n\rightarrow\infty} \mathbb{E}\int_0^t \Vert g_n-g\Vert^2_{L^2(D)}\, ds=0.
\end{align}
Since $u_n(\omega,s)\rightarrow u(\omega,s)$ for $n\rightarrow\infty$ in $L^1(D)$ and
\begin{align*}
\Vert u_n(\omega,s)\Vert_{L^1(D)}\leq C_U\Vert u(\omega,s)\Vert_{L^1(D)}
\end{align*} 
for a.e. $(\omega,s)\in \Omega\times (0,T)$ and all $n\in\mathbb{N}$, from Lebesgue's dominated convergence theorem it follows that $u_n\rightarrow u$ in $L^1(\Omega\times Q_T)$ and, passing to a not relabeled subsequence if necessary, also a.s. in $\Omega\times Q_T$. Consequently, a.s. in $\Omega\times Q_T$, we get
\begin{align*}
\lim_{n\rightarrow\infty}|S'(u_n(\omega,s,x))-S'(u(\omega,s,x))|^2|g(\omega,s,x)|^2=0.
\end{align*}
In addition,
\begin{align*}
|S'(u_n(\omega,s,x))-S'(u(\omega,s,x))|^2|g(\omega,s,x)|^2\leq 2\Vert S'\Vert_{\infty}^2|g(\omega,s,x)|^2
\end{align*}
a.s. in $\Omega\times Q_T$ and from Lebesgue's dominated convergence theorem it follows that
\begin{align}\label{190318_04}
\lim_{n\rightarrow\infty} E\int_0^t\int_D|S'(u_n)-S'(u)|^2|g|^2\,dx\,ds=0
\end{align}
for any $t\in [0,T]$. Combining \eqref{190318_01}, \eqref{190318_03} and \eqref{190318_04}, it follows that
\begin{align*}
\lim_{n\rightarrow\infty}\int_0^t\int_DS'(u_n)\psi g_n\,dx\,d\beta=\int_0^t\int_DS'(u)\psi g \,dx\,d\beta
\end{align*}
in $L^2(\Omega)$ for any $t\in [0,T]$, and, passing a not relabeled subsequence if necessary, also a.s. in $\Omega$. Hence, up to a not relabeled subsequence,
\begin{align}\label{190318_05}
\lim_{n\rightarrow\infty}I_4=\int_0^t\int_D S(u)\psi g\,dx\,d\beta
\end{align}
a.s. in $\Omega$. From the boundedness of $S$ and the convergence of $u_n(\omega,s)$ to $u(\omega,s)$ in $L^1(D)$ for all $s\in(0,t)$, a.s. in $\Omega$, it follows that $S(u_n)\rightarrow S(u)$ for $n\rightarrow\infty$ in $L^1(Q_T)$, a.s. in $\Omega$ and therefore
\begin{align}\label{190318_06}
\lim_{n\rightarrow\infty}I_5=\int_0^t S(u)\psi_t\,dx\,ds
\end{align} 
a.s. in $\Omega$. According to the convergence properties of $(g_n)$, $g_n^2\rightarrow g^2$ in $L^1((0,t)\times D)$ for $n\rightarrow\infty$ a.s. in $\Omega$. On the other hand, from the boundedness and the continuity of $S''$ we get $S''(u_n)\rightarrow S''(u)$ in $L^q((0,t)\times D)$ for all $1\leq q<\infty$ and weak-$\ast$ in $L^{\infty}((0,t)\times D)$ a.s. in $\Omega$, thus it follows that
\begin{align}\label{190318_07}
\lim_{n\rightarrow\infty} I_6=\int_0^t\int_DS''(u)\psi g^2\,dx\,ds
\end{align}
a.s. in $\Omega$.
Summarizing our results in \eqref{190314_01}, \eqref{190315_02}, \eqref{190315_04}, \eqref{190318_05}, \eqref{190318_06} and \eqref{190318_07}, we get
\begin{align*}
&\int_D S(u(t))\psi(t) - S(u_0) \psi(0) \,dx\\
&+\int_0^t \langle (-\operatorname{div} \,G+f), S'(u)\psi \rangle_{W^{-1,p'}(D)+L^1(D), W_0^{1,p}(D)\cap L^{\infty}(D)} \,ds\\
= &\int_0^t \int_D S'(u) \psi g \,dx\,d\beta + \int_0^t \int_D S(u) \psi_t \,dx\,ds + \frac{1}{2} \int_0^t \int_D  S''(u) \psi g^2 \,dx\,ds
\end{align*}
for all $t \in [0,T]$ and a.s. in $\Omega$. 

\end{proof}

\section{Renormalized solutions}
Let us assume that there exists a strong solution $u$ to \eqref{1} in the sense of Theorem \ref{Theorem 2.1}. We observe that for initial data $u_0$ merely in $L^1$, the It\^{o} formula for the square of the norm (see, e.g., \cite{EP}) can not be applied and consequently the natural a priori estimate for $\nabla u$ in $L^p(\Omega\times Q_T)^d$ is not available. Choosing $g=\Phi$, $f\equiv 0$, $\psi\equiv 1$ and 
\[S(u)=\int_0^{u}T_k(r)\, dr\]
in \eqref{Itoformula}, where $T_k:\mathbb{R}\rightarrow\mathbb{R}$ is the truncation function at level $k>0$ defined by
\begin{align*}
T_k(r)=\begin{cases} r &,~|r| \leq k, \\ 
k\operatorname{sign}(r) &,~ |r|> k,
\end{cases}
\end{align*}
we find that there exists a constant $C(k)\geq 0$ depending on the truncation level $k>0$, such that
\[\mathbb{E}\int_0^T\int_D |\nabla T_k(u)|^p\,dx\,ds\leq C(k).\]
As in the deterministic case, the notion of renormalized solutions takes this information into account : 
\begin{definition}\label{Definition 5.1}
Let $u_0 \in L^1(\Omega \times D)$ be $\mathcal{F}_0$-measurable. An $\mathcal{F}_t$-adapted stochastic process $u: \Omega \times [0,T] \to L^1(D)$ such that $u \in L^1(\Omega; \mathcal{C}([0,T];L^1(D)))$ is a renormalized solution to (\ref{1}) with initial value $u_0$, if and only if
\begin{itemize}
\item[(i)]
$T_k(u) \in L^p(\Omega; L^p(0,T;W_0^{1,p}(D)))$ for all $k>0$. 
\item[(ii)]
For all $\psi \in \mathcal{C}^{\infty}([0,T] \times \bar{D})$ and all $S \in \mathcal{C}^2(\mathbb{R})$ such that $S'$ has compact support with $S'(0)=0$ or $\psi(t,x) =0$ for all $(t,x) \in [0,T] \times \partial D$ the equality
\begin{align}\label{reneq190204}
&\int_D S(u(t))\psi(t) - S(u_0) \psi(0) \,dx + \int_0^t \int_D S''(u) |\nabla u|^p \psi \,dx\,ds\nonumber \\
+ &\int_0^t \int_D S'(u) |\nabla u|^{p-2} \nabla u \cdot \nabla \psi \,dx\,ds\nonumber\\
= &\int_0^t \int_D S'(u) \psi \Phi \,dx\,d\beta + \int_0^t \int_D S(u) \psi_t \,dx\,ds + \frac{1}{2} \int_0^t \int_D  S''(u) \psi \Phi^2 \,dx\,ds
\end{align}
holds true for all $t \in [0,T]$ and a.s. in $\Omega$.
\item[(iii)]
The following energy dissipation condition holds true:
\begin{align*}
\lim\limits_{k \to \infty}\mathbb{E}\int_{\{k < |u| < k+1 \}} |\nabla u|^p \,dx\,dt= 0.
\end{align*}
\end{itemize}
\end{definition}
Several remarks about Definition \ref{Definition 5.1} are in order:
Let $u$ be a renormalized solution in the sense of Definition \ref{Definition 5.1}. Since $\operatorname{supp}\,(S')\subset [-M,M]$, it follows that $S$ is constant outside $[-M,M]$ and for all $k\geq M$, $S(u(t))=S(T_k(u(t)))$ a.s. in $\Omega\times D$ for all $t\in [0,T]$. In particular, we have 
\[S(u)\in L^1(\Omega;\mathcal{C}([0,T];L^1(D)))\cap L^p(\Omega;L^p(0,T;W^{1,p}(D)))\cap L^{\infty}(\Omega\times Q_T).\] 
From the chain rule for Sobolev functions it follows that
\begin{equation}\label{180701_02}
S'(u)(|\nabla u|^{p-2}\nabla u)=S'(T_M(u))(|\nabla T_M(u)|^{p-2}\nabla T_M(u))=S'(u)\chi_{\{|u|<M\}}(|\nabla u|^{p-2}\nabla u)
\end{equation}
a.s. in $\Omega\times Q_T$ and therefore from $(i)$ it follows that all the terms in \eqref{reneq190204} are well-defined. In general, for the  renormalized solution $u$, $\nabla u$ may not be in $L^p(\Omega\times Q_T)^d$ and therefore $(iii)$ is an additional condition which can not be derived from $(ii)$. However, for $u\in L^1(\Omega\times Q_T)$ satisfying $(i)$, we can define a generalized gradient (still denoted by $\nabla u$) by setting
\[\nabla u(\omega,t,x):=\nabla T_k(u)\]
a.s. in $ \{ |u| < k\}$ for $k>0$. The function $\nabla u$ is well defined since $\bigcup\limits_{k>0}  \{ |u| < k\} = \Omega \times Q_T$, $T_k(u)=T_k(T_{k+\epsilon}(u))$, $T_{k + \epsilon}(u) \in L^p(\Omega;L^p(0,T;W^{1,p}(D)))$ and therefore
\begin{align*}
\nabla T_k(u)= \nabla T_k(T_{k + \epsilon}(u)) = \nabla T_{k + \epsilon}(u) \chi_{\{|u|<k\}} = \nabla T_{k+ \epsilon}(u)
\end{align*}
in $ \{ |u| < k\}$ for all $k, \epsilon >0$.
For $u\in L^1(\Omega;\mathcal{C}([0,T];L^1(D)))$ such that $T_k(u)\in L^p(\Omega;L^p(0,T;W^{1,p}_0(D)))$ for all $k>0$, $(ii)$ is equivalent to
\begin{align}\label{180625_01}
&S(u(t))-S(u(0))-\int_0^t \operatorname{div}\,(S'(u)|\nabla u|^{p-2}\nabla u) \,ds+\int_0^t S''(u)[|\nabla u|^p-\frac{1}{2}\Phi^2] \,ds \nonumber\\
&=\int_0^t \Phi S'(u) \,d\beta,
\end{align}
or equivalently, in differential form,
\begin{align}\label{SPDE1}
&dS(u)-\operatorname{div}\,(S'(u)|\nabla u|^{p-2}\nabla u) \,dt+S''(u)[|\nabla u|^p-\frac{1}{2}\Phi^2]\,dt
=\Phi S'(u) \,d\beta
\end{align}
in $W^{-1,p'}(D)+L^1(D)$ for all $t\in[0,T]$, a.s. in $\Omega$ and for any $S\in \mathcal{C}^2(\mathbb{R})$ with $\operatorname{supp}(S')$ compact, and, since the right-hand side of \eqref{SPDE1} is in $L^2(D)$, also in $L^2(D)$.
\begin{remark}\label{Remark 5.2}
If $u$ is a renormalized solution to \eqref{1}, thanks to \eqref{SPDE1}, the It\^{o} formula from Proposition \ref{Proposition 4.2} still holds true for $S(u)$ for any $S\in \mathcal{C}^2(\mathbb{R})$ with $\operatorname{supp}(S')$ compact such that $S(u)\in W^{1,p}_0(D)$ a.s. in $\Omega\times (0,T)$. Indeed, in this case \eqref{2} is satisfied for the progressively measurable functions
\[\tilde{u}=S(u)\in L^1(\Omega; \mathcal{C}([0,T];L^1(D))) \cap L^p(\Omega; L^p(0,T;W_0^{1,p}(D))),\] 
\[G=S'(u)|\nabla u|^{p-2}\nabla u\in L^{p'}(\Omega\times Q_T)^d,\]
\[f=S''(u)[|\nabla u|^p-\frac{1}{2}\Phi^2\in L^1(\Omega\times Q_T),\]
\[g=\Phi S'(u)\in L^2(\Omega\times Q_T).\] 
\end{remark}
\begin{remark}
Let $u$ be a renormalized solution to \eqref{1} with $\nabla u\in L^p(\Omega\times Q_T)^d$. For fixed $l>0$, let $h_l:\mathbb{R}\rightarrow\mathbb{R}$ be defined by
\begin{align*}
h_l(r)=\begin{cases} 0 &,~|r| \geq l+1 \\ 
l+1-|r| &,~ l<|r|<l+1\\
1 &,~ |r|\leq l.
\end{cases}
\end{align*}
Taking  $S(u)=\int_0^{u} h_l(r) \,dr$ as a test function in \eqref{181201_04}, we may pass to the limit with $l\rightarrow\infty$ and we find that $u$ is a strong solution to \eqref{1}.
\end{remark}

\section{Existence of renormalized solutions}

Before we show the existence of a renormalized solution, we show the following lemma.
\begin{lemma}\label{Lemma 6.1}
Let $(u_0^n)_n \subset L^2(\Omega \times D)$ be an $\mathcal{F}_0$-measurable sequence such that $u_0^n \to u_0$ in $L^1(\Omega \times D)$ for a function $u_0 \in L^1(\Omega \times D)$. Furthermore, let $u_n$ be a strong solution to \eqref{1} with respect to the initial value $u_0^n$. Then there exists an $\mathcal{F}_t$-adapted stochastic process $u: \Omega \times [0,T] \to L^1(D)$ such that $u \in L^1(\Omega; \mathcal{C}([0,T]; L^1(D)))$ and a subsequence in $n$ such that
\begin{align*}
u_n \to u ~~~\textnormal{in}~L^1(\Omega; \mathcal{C}([0,T]; L^1(D)))~\textnormal{and in}~\mathcal{C}([0,T]; L^1(D))~\textnormal{a.s. in}~\Omega.
\end{align*}
\end{lemma}
\begin{proof}
By assumption $(u_0^n)_n$ is a Cauchy sequence in $L^1(\Omega \times D)$. Then, Theorem \ref{Theorem 3.1} yields
\begin{align*}
\sup\limits_{t \in [0,T]} \int_D |u_n(t) - u_m(t)| \,dx \leq \int_D |u_0^n - u_0^m| \,dx \to 0
\end{align*}
as $n,m \to \infty$, a.s. in $\Omega$ and in $L^1(\Omega)$. Especially, a.s. in $\Omega$ and for all $t \in [0,T]$ we have $u_n(t) \to u(t)$ in $L^1(D)$. As a limit function of a sequence of $\mathcal{F}_t$-measurable functions we may conclude that $u(t)$ is $\mathcal{F}_t$-measurable.
\end{proof}
\begin{theorem}\label{Theorem 6.2}
Let the assumptions in Definition \ref{Definition 5.1} be satisfied. Then there exists a renormalized solution to \eqref{1}.
\end{theorem}
\begin{proof}\renewcommand{\qedsymbol}{}
Let $(u_0^n)_n \subset L^2(\Omega \times D)$ be an $\mathcal{F}_0$-measurable sequence such that $u_0^n \to u_0$ in $L^1(\Omega \times D)$. Now, let $u_n$ be a strong solution to (\ref{1}) with initial value $u_0^n$, i.e.,
\begin{align}\label{20}
u_n(t) - u_0^n - \int_0^t \textnormal{div}\, (|\nabla u_n|^{p-2} \nabla u_n) \,ds = \int_0^t \Phi \,d\beta
\end{align}
for all $t \in [0,T]$ and a.s. in $\Omega$. By Lemma \ref{Lemma 6.1} there exists an $\mathcal{F}_t$-adapted stochastic process $u: \Omega \times [0,T] \to L^1(D)$ such that $u \in L^1(\Omega; \mathcal{C}([0,T]; L^1(D)))$ and a subsequence in $n$ such that
\begin{align*}
u_n \to u ~~~\textnormal{in}~L^1(\Omega; \mathcal{C}([0,T]; L^1(D)))~\textnormal{and in}~\mathcal{C}([0,T]; L^1(D))~\textnormal{a.s. in}~\Omega.
\end{align*}
We claim that this function $u$ is a renormalized solution to \eqref{1} with initial value $u_0$.\\
\newline
Firstly, we apply the It\^{o} formula introduced in Proposition \ref{Proposition 4.2} to equality \eqref{20}. Therefore we know that for all $\psi \in C^{\infty}([0,T] \times \overline{D})$ and all $S \in W^{2,\infty}(\mathbb{R})$ such that $S''$ is piecewise continuous and $S'(0)=0$ or $\psi(t,x) =0$ for all $(t,x) \in [0,T] \times\partial D$ the equality
\begin{align}\label{190206_03}
&\int_D S(u_n(t))\psi(t) - S(u_0^n) \psi(0) \,dx + \int_0^t \int_D S''(u_n) |\nabla u_n|^p \psi \,dx\,ds \notag \\
&+ \int_0^t \int_D S'(u_n) |\nabla u_n|^{p-2} \nabla u_n \cdot \nabla \psi \,dx\,ds\\
= &\int_0^t \int_D S'(u_n) \psi \Phi \,dx\,d\beta + \int_0^t \int_D S(u_n) \psi_t \,dx\,ds + \frac{1}{2} \int_0^t \int_D  S''(u_n) \psi \Phi^2 \,dx\,ds \notag
\end{align}
holds true for all $t \in [0,T]$ and a.s. in $\Omega$.\\
Now, we plug $S(r)=\int_0^{r} T_k(\overline{r}) \,d\overline{r}$  into \eqref{190206_03}.
Using $\psi=1$ we get
\begin{align*}
&\int_D \int_{u_0^n}^{u_n(t)} T_k(r) \,dr\,dx + \int_0^t \int_D \chi_k(u_n) |\nabla u_n|^p \,dx\,ds\\
= &\int_0^t \int_D T_k(u_n) \Phi \,dx\,d\beta + \frac{1}{2} \int_0^t \int_D  \chi_k(u_n) \Phi^2 \,dx\,ds
\end{align*}
for all $k>0$, all $t \in [0,T]$ and a.s. in $\Omega$. Taking expectation yields
\begin{align*}
&\mathbb{E} \int_D \int_{u_0^n}^{u_n(t)} T_k(r) \,drdx + \mathbb{E}\int_0^t \int_D |\nabla T_k(u_n)|^p \,dx\,ds\\
= \frac{1}{2} &\mathbb{E} \int_0^t \int_D  \chi_k(u_n) \Phi^2 \,dx\,ds
\end{align*}
for all $k>0$, all $t \in [0,T]$ and a.s. in $\Omega$. The first term on the left hand side is nonnegative. Since $|\chi_k| \leq 1$ we may conclude that $T_k(u_n)$ is bounded in $L^p(\Omega; L^p(0,T;W_0^{1,p}(D)))$ for all $k>0$.\\
Hence for a subsequence we have $T_k(u_n) \rightharpoonup T_k(u)$ in $L^p(\Omega; L^p(0,T;W_0^{1,p}(D)))$ for all $k>0$, which claims (i).\\
Furthermore, $ |\nabla T_k(u_n)|^{p-2} \nabla T_k(u_n) $ is bounded in $L^{p'}(\Omega \times Q_T)^d$. Consequently, there exists a not relabeled subsequence of $n$ such that $|\nabla T_k(u_n)|^{p-2} \nabla T_k(u_n) \rightharpoonup \sigma_k$ in $L^{p'}(\Omega \times Q_T)^d$.\\
\newline
Obviously, the proof of (ii) is done as far as we can show that 
\begin{align*}
T_k(u_n) \to T_k(u)~~~ \textnormal{in} ~~~L^p(\Omega; L^p(0,T;W_0^{1,p}(D)))
\end{align*}
for all $k>0$. 
This will be done in the following lemma that is inspired by Theorem 2 of \cite{DB}.
\end{proof}
\begin{lemma}\label{Lemma 6.3}
Under the assumptions of Theorem \ref{Theorem 6.2} we have
\begin{align}\label{22}
\lim\limits_{n,m \to \infty} \mathbb{E} &\int_0^T \int_D \bigg( |\nabla T_k(u_n)|^{p-2} \nabla T_k(u_n) - |\nabla T_k(u_m)|^{p-2} \nabla T_k(u_m) \bigg) \cdot  \\
&\cdot( \nabla T_k(u_n) - \nabla T_k(u_m)) \,dx\,ds=0, \notag
\end{align}
where $u_n$ is a strong solution to \eqref{1} with initial value $u_0^n \in L^2(\Omega \times D)$ satisfying $u_0^n \to u_0$ in $L^1(\Omega \times D)$.\\
Especially, we have
\begin{align*}
\nabla T_k(u_n) \to \nabla T_k(u)~~~ \textnormal{in} ~~~L^p(\Omega \times Q_T)^d
\end{align*}
and
\begin{align*}
T_k(u_n) \to T_k(u)~~~ \textnormal{in} ~~~L^p(\Omega; L^p(0,T;W_0^{1,p}(D)))
\end{align*}
for $n \to \infty$ and for all $k>0$.
\end{lemma}
\begin{remark}\label{Remark 6.4}
In Lemma \ref{Lemma 6.3} and in the following we use the notation $\lim\limits_{n,m \to \infty} F_{n,m}$ if $n$ and $m$ tend successively to $\infty$ and
\begin{align*}
\lim\limits_{m \to \infty} \lim\limits_{n \to \infty} F_{n,m} = \lim\limits_{n \to \infty} \lim\limits_{m \to \infty} F_{n,m}.
\end{align*}
\end{remark}
\begin{proof}
Since $u_n$ and $u_m$ are strong solutions to \eqref{1}, we consider the difference of the corresponding equations. Using $T_k(u_n - u_m)$ as a test function it yields
\begin{align*}
&\int_D \tilde{T}_k(u_n(T) - u_m(T)) \,dx \\
+ &\int_0^T \int_D (|\nabla u_n|^{p-2} \nabla u_n - |\nabla u_m|^{p-2} \nabla u_m) \cdot \nabla T_k(u_n - u_m) \,dx\,dt\\
= &\int_D \tilde{T}_k(u_0^n - u_0^m) \,dx
\end{align*}
a.s. in $\Omega$ and for all $k>0$, where $\tilde{T}_k(s) := \int_0^s T_k(r) \,dr$ for all $s \in \mathbb{R}$. Since $\tilde{T}_k$ is nonnegative we may conclude that
\begin{align}\label{23}
\lim\limits_{n,m \to \infty} \mathbb{E} \int_0^T \int_D (|\nabla u_n|^{p-2} \nabla u_n - |\nabla u_m|^{p-2} \nabla u_m) \cdot \nabla T_k(u_n - u_m) \,dx\,dt=0
\end{align}
for all $k>0$. We set 
\begin{align*}
&\int_0^T \int_D \bigg( |\nabla T_k(u_n)|^{p-2} \nabla T_k(u_n) - |\nabla T_k(u_m)|^{p-2} \nabla T_k(u_m) \bigg) \cdot  \\
&\cdot( \nabla T_k(u_n) - \nabla T_k(u_m)) \,dx\,dt \\
&= I_k^{n,m} + J_k^{n,m} + J_k^{m,n},
\end{align*}
a.s. in $\Omega$, where
\begin{align*}
I_k^{n,m} &= \int_{\{|u_n| \leq k\} \cap \{|u_m| \leq k\}} (|\nabla u_n|^{p-2} \nabla u_n - |\nabla u_m|^{p-2} \nabla u_m) \cdot \nabla (u_n - u_m) \,dx\,dt, \\
J_k^{n,m} &= \int_{\{|u_n| \leq k\} \cap \{|u_m| > k\}} |\nabla u_n|^{p-2} \nabla u_n \cdot \nabla u_n \,dx\,dt
\end{align*}
a.s. in $\Omega$. $J_k^{m,n}$ is the same as $J_k^{n,m}$ where the roles of $n$ and $m$ are reversed. Therefore these two terms can be treated simultaneously.\\
Since $ \{|u_n| \leq k\} \cap \{|u_m| \leq k\} \subset \{ |u_n - u_m| \leq 2k \}$, we get
\begin{align*}
0 &\leq \lim\limits_{n,m \to \infty} \mathbb{E} I_k^{n,m} \\
&\leq \lim\limits_{n,m \to \infty} \mathbb{E} \int_0^T \int_D (|\nabla u_n|^{p-2} \nabla u_n - |\nabla u_m|^{p-2} \nabla u_m) \cdot \nabla T_{2k}(u_n - u_m) \,dx\,dt =0
\end{align*}
for all $k>0$ by \eqref{23}. Now we set
\begin{align*}
0 \leq J_k^{n,m} = J_{1,k,k'}^{n,m} + J_{2,k,k'}^{n,m},
\end{align*}
where
\begin{align*}
J_{1,k,k'}^{n,m} &= \int_{\{|u_n| \leq k\} \cap \{|u_m| > k\} \cap \{|u_n - u_m| \leq k'\}} |\nabla u_n|^{p-2} \nabla u_n \cdot \nabla u_n \,dx\,dt,\\
J_{2,k,k'}^{n,m} &=\int_{\{|u_n| \leq k\} \cap \{|u_m| > k\} \cap \{|u_n - u_m| > k'\}} |\nabla u_n|^{p-2} \nabla u_n \cdot \nabla u_n \,dx\,dt
\end{align*}
for all $k'>k>0$, a.s. in $\Omega$. Firstly, we focus on $J_{1,k,k'}^{n,m}$. It is $ \{|u_n| \leq k\} \cap \{|u_m| > k\} \cap \{|u_n - u_m| \leq k'\} \subset \{|u_n| \leq k\} \cap \{k < |u_m| \leq k+k' \}$. Therefore we can estimate
\begin{align*}
0 &\leq J_{1,k,k'}^{n,m} \\
&\leq \int_{\{|u_n| \leq k\} \cap \{k < |u_m| \leq k+k' \}} (|\nabla u_n|^{p-2} \nabla u_n - |\nabla u_m|^{p-2} \nabla u_m) \cdot \nabla (u_n - u_m) \,dx\,dt \\
&+ \int_{\{|u_n| \leq k\} \cap \{k < |u_m| \leq k+k' \}} |\nabla u_n|^{p-2} \nabla u_n \nabla u_m \,dx\,dt\\
&+ \int_{\{|u_n| \leq k\} \cap \{k < |u_m| \leq k+k' \}} |\nabla u_m|^{p-2} \nabla u_m \nabla u_n \,dx\,dt\\
&= J_{1,1,k,k'}^{n,m} +J_{1,2,k,k'}^{n,m} + J_{1,3,k,k'}^{n,m}
\end{align*}
a.s. in $\Omega$. We see that \eqref{23} yields
\begin{align*}
0 &\leq \lim\limits_{n,m \to \infty} \mathbb{E} J_{1,1,k,k'}^{n,m} \\
&\leq \lim\limits_{n,m \to \infty} \mathbb{E} \int_0^T \int_D (|\nabla u_n|^{p-2} \nabla u_n - |\nabla u_m|^{p-2} \nabla u_m) \cdot \nabla T_{2k + k'}(u_n - u_m) \,dx\,dt = 0.
\end{align*}
It is
\begin{align*}
\mathbb{E} J_{1,2,k,k'}^{n,m} = \mathbb{E} \int_0^T \int_D |\nabla T_k(u_n)|^{p-2} \nabla T_k(u_n) \cdot \nabla T_{k+k'}(u_m) \chi_{\{k < |u_m| \leq k+k' \}} \,dx\,dt.
\end{align*}
Let us define $\theta_k^{k'}(r):= T_{k+k'}(r) - T_k(r)$. \\
Then $\nabla \theta_k^{k'}(u_m) = \nabla T_{k+k'}(u_m) \chi_{\{k < |u_m| < k+k' \}}$ and $\nabla \theta_k^{k'}(u_m) \rightharpoonup \nabla \theta_k^{k'}(u)$ in $L^{p}(\Omega \times Q_T)^d$. Now we can estimate
\begin{align*}
0 \leq \lim\limits_{n,m \to \infty} \mathbb{E} J_{1,2,k,k'}^{n,m} = \mathbb{E} \int_0^T \int_D \sigma_k \nabla \theta_k^{k'}(u) \,dx\,dt.
\end{align*}
We show that $\sigma_k= \chi_{\{|u| < k\}} \sigma_{k+1}$ a.e. on $\{|u| \neq k \}$. If we do so it follows that $\sigma_k=0$ a.e. on $\{|u| > k \}$. Since $\nabla \theta_k^{k'}(u)=0$ on $\{ |u| \leq k \}$ it follows $\lim\limits_{n,m \to \infty} \mathbb{E} J_{1,2,k,k'}^{n,m}=0$.\\
Let $\psi \in L^p(\Omega \times Q_T)^d$. Then
\begin{align*}
\lim\limits_{n \to \infty} \mathbb{E} \int_{Q_T} |\nabla T_k(u_n)|^{p-2} \nabla T_k(u_n) \cdot \psi \cdot \chi_{\{|u| \neq k \}} \,dx\,dt = \mathbb{E} \int_{Q_T} \sigma_k \psi \cdot \chi_{\{|u| \neq k \}} \,dx\,dt.
\end{align*}
On the other hand we know that $u_n \to u$ a.e. in $\Omega \times Q_T$. Hence, we have $\chi_{\{|u_n|<k \}} \to \chi_{\{|u|<k \}}$ a.e. in $\{|u| \neq k \}$. Therefore the theorem of Lebesgue yields
\begin{align*}
\chi_{\{|u_n|<k \}} \cdot \chi_{\{|u| \neq k \}} \cdot \psi \to \chi_{\{|u|<k \}} \cdot \chi_{\{|u|\neq k \}} \cdot \psi ~~~\textnormal{in}~ L^p(\Omega \times Q_T)^d.
\end{align*}
We may conclude that
\begin{align*}
&\lim\limits_{n \to \infty} \mathbb{E} \int_{Q_T} |\nabla T_k(u_n)|^{p-2} \nabla T_k(u_n) \cdot \psi \cdot \chi_{\{|u| \neq k \}} \,dx\,dt  \\
= &\lim\limits_{n \to \infty} \mathbb{E} \int_{Q_T} |\nabla T_{k+1}(u_n)|^{p-2} \nabla T_{k+1}(u_n) \cdot \psi \cdot \chi_{\{|u| \neq k \}} \cdot \chi_{\{|u_n|<k \}} \,dx\,dt \\
= & \mathbb{E} \int_{Q_T} \sigma_{k+1} \psi \chi_{\{|u| \neq k \}} \cdot \chi_{\{|u|<k \}} \,dx\,dt.
\end{align*}
It follows that $\sigma_k= \chi_{\{|u| < k\}} \sigma_{k+1}$ a.e. on $\{|u| \neq k \}$ and therefore 
\begin{align*}
\lim\limits_{n,m \to \infty} \mathbb{E} J_{1,2,k,k'}^{n,m}=0.
\end{align*}
Now let us consider $ J_{1,3,k,k'}^{n,m}$. Since $|\nabla \theta_k^{k'}(u_n)|^{p-2} \nabla \theta_k^{k'}(u_n) \rightharpoonup \tilde{\sigma}_k^{k'}$ in $L^{p'}(\Omega \times Q_T)$ for a subsequence we have
\begin{align*}
\lim\limits_{n,m \to \infty} \mathbb{E} J_{1,3,k,k'}^{n,m} = \mathbb{E} \int_0^T \int_D \tilde{\sigma}_k^{k'} \cdot \nabla T_k(u) \,dx\,dt.
\end{align*}
Since $\nabla \theta_k^{k'}(v)= \chi_{\{k < |v| < k+k'\}} \nabla \theta_{k-1}^{k'+2}(v)$ for all $v \in L^p(\Omega; L^p(0,T;W_0^{1,p}(D)))$, we can show by similar arguments as before that $\tilde{\sigma}_k^{k'}=  \chi_{\{k < |u| < k+k'\}}\tilde{\sigma}_{k-1}^{k'+2}$ a.e. on $\{|u| \neq k\} \cup \{|u| \neq k+k'\}$. Therefore it follows that
\begin{align*}
\lim\limits_{n,m \to \infty} \mathbb{E} J_{1,3,k,k'}^{n,m}=0.
\end{align*}
It is left to show that
\begin{align*}
\lim\limits_{n,m \to \infty} \mathbb{E} J_{2,k,k'}^{n,m}=0.
\end{align*}
This is a consequence of the following lemma, that is similar to Lemma 2 in \cite{DB} (see also Theorem 2 in \cite{DB}).
\begin{lemma}\label{Lemma 5.6}
Let $H$ and $Z$ be two real valued functions belonging to $W^{2,\infty}(\mathbb{R})$ such that $H''$ and $Z''$ are piecewise continuous, $H'$ and $Z'$ have compact supports and $Z(0)=Z'(0)=0$ is satisfied. Then
\begin{align}\label{7}
\lim\limits_{n,m \to \infty} \mathbb{E} \int_0^T \int_D H''(u_n) Z(u_n - u_m) |\nabla u_n|^p \,dx\,dt =0.
\end{align}
\end{lemma}
\begin{proof}
Using the product rule for the It\^{o} formula (see Proposition \ref{itoproduct}) yields
\begin{align*}
&\int_D Z(u_n(t) - u_m(t)) H(u_n(t)) \,dx= \int_D Z(u_0^n - u_0^m)H(u_0^n) \,dx \\
- &\int_0^t \int_D |\nabla u_n|^{p-2} \nabla u_n \nabla \bigg(Z(u_n - u_m) H'(u_n) \bigg) \,dx\,ds \\
+ &\frac{1}{2} \int_0^t \int_D H''(u_n)Z(u_n - u_m) \Phi^2 \,dx\,ds \\
+ &\int_0^t \int_D H'(u_n)Z(u_n - u_m) \Phi \,dx\,d\beta \\
- &\int_0^t \int_D (|\nabla u_n|^{p-2} \nabla u_n - |\nabla u_m|^{p-2} \nabla u_m) \nabla \bigg( Z'(u_n - u_m)H(u_n) \bigg) \,dx\,ds
\end{align*}
for all $t \in [0,T]$ and a.s. in $ \Omega$. Using $t=T$ and passing to the limit yields
\begin{align*}
&\lim\limits_{n,m \to \infty} \mathbb{E} \int_0^T \int_D H''(u_n) Z(u_n - u_m) |\nabla u_n|^p \,dx\,dt\\
= -&\lim\limits_{n,m \to \infty}  L^{n,m} -\lim\limits_{n,m \to \infty} M^{n,m} -\lim\limits_{n,m \to \infty} N^{n,m},
\end{align*}
where
\begin{align*}
L^{n,m} &=  \mathbb{E} \int_0^T \int_D Z''(u_n - u_m) H(u_n)(|\nabla u_n|^{p-2} \nabla u_n - |\nabla u_m|^{p-2} \nabla u_m) \cdot \\
&\cdot \nabla (u_n - u_m) \,dx\,dt, \\
M^{n,m} &=  \mathbb{E} \int_0^T \int_D Z'(u_n - u_m) H'(u_n)(|\nabla u_n|^{p-2} \nabla u_n - |\nabla u_m|^{p-2} \nabla u_m) \cdot \nabla u_n \,dx\,dt, \\
N^{n,m} &=  \mathbb{E} \int_0^T \int_D Z'(u_n - u_m) H'(u_n)|\nabla u_n|^{p-2} \nabla u_n \cdot \nabla (u_n - u_m) \,dx\,dt.
\end{align*}
The rest of the proof is the same as the proof of \cite{DB}, Theorem 2.
\end{proof}
We continue the proof of Lemma \ref{Lemma 6.3}. Using
\begin{align*}
(H_k^{\delta})''(r)=
\begin{cases}
1, ~&|r|<k, \\
-k \delta, ~&k \leq |r| \leq k+ \frac{1}{\delta}, \\
0, ~&|r| > k + \frac{1}{\delta}
\end{cases}
\end{align*}
in equality \eqref{7} yields
\begin{align*}
&\limsup\limits_{n\to \infty} \limsup\limits_{m \to \infty}  \mathbb{E}\int_{\{|u_n| \leq k\}} Z(u_n - u_m) |\nabla u_n|^p \,dx\,dt\,\\
\leq &k \delta \limsup\limits_{n\to \infty} \limsup\limits_{m \to \infty}  \mathbb{E} \int_{\{k \leq |u_n| \leq k+\frac{1}{\delta}\}} Z(u_n - u_m) |\nabla u_n|^p \,dx\,dt\, \\
\leq &\delta \cdot k \Vert Z \Vert_{\infty} \limsup\limits_{n \to \infty}  \mathbb{E} \int_{\{k \leq |u_n| \leq k+\frac{1}{\delta}\}} |\nabla u_n|^p \,dx\,dt\,.
\end{align*}
Now applying Proposition \ref{Proposition 4.2} with $S= \int_0^{\cdot} \theta_{k}^{\frac{1}{\delta}} =: \tilde{\theta}_k^{\frac{1}{\delta}}$, $\psi= 1$, $g=\Phi$ and $f=0$ and taking expectation yields
\begin{align*}
&\mathbb{E} \int_D \tilde{\theta}_k^{\frac{1}{\delta}}(u_n(T)) \,dx + \mathbb{E} \int_0^T \int_D \chi_{\{k \leq |u_n| \leq u_n + \frac{1}{\delta}\}} |\nabla u_n|^p \,dx \,dt \\
= &\mathbb{E} \int_D \tilde{\theta}_k^{\frac{1}{\delta}}(u_0^n) \,dx + \frac{1}{2} \mathbb{E} \int_0^T \int_D \chi_{\{k \leq |u_n| \leq u_n + \frac{1}{\delta}\}} \Phi^2 \,dx \, dt.
\end{align*}
The first term on the left hand side is nonnegative and the integrand of the second term on the right hand side can be estimated as follows
\begin{align*}
 \chi_{\{k \leq |u_n| \leq u_n + \frac{1}{\delta}\}} \Phi^2 \leq \Phi^2 ~\in ~L^1(\Omega \times Q_T).
\end{align*}
Multiplying by $\delta$ and passing to the limit with $n \to \infty$ yields
\begin{align*}
\delta \cdot \limsup\limits_{n \to \infty}~ \mathbb{E} \int_0^T \int_D \chi_{\{k \leq |u_n| \leq u_n + \frac{1}{\delta}\}} |\nabla u_n|^p \,dx \,dt \leq \mathbb{E} \int_D \delta \tilde{\theta}_k^{\frac{1}{\delta}}(u_0) \,dx + \frac{1}{2} \delta \Vert \Phi \Vert_{L^2(\Omega \times Q_T)}^2.
\end{align*}
We can estimate that $\delta \tilde{\theta}_k^{\frac{1}{\delta}}(u_0) \to 0$ a.e. in $\Omega \times D$ as $\delta \to 0$ and $|\delta \tilde{\theta}_k^{\frac{1}{\delta}}(u_0)| \leq u_0 + C$ for a constant $C>0$. Therefore Lebesgue's Theorem yields
\begin{align*}
\lim\limits_{\delta \to 0} \limsup\limits_{n \to \infty} \delta \cdot \mathbb{E} \int_{\{k \leq |u_n| \leq k+\frac{1}{\delta}\}} Z(u_n - u_m) |\nabla u_n|^p \,dx\,dt\, =0. 
\end{align*}
Therefore we may conclude
\begin{align*}
&\lim\limits_{n,m \to \infty}  \mathbb{E}\int_{\{|u_n| \leq k\}} Z(u_n - u_m) |\nabla u_n|^p \,dx\,dt\,\\
= &\limsup\limits_{n\to \infty} \limsup\limits_{m \to \infty}  \mathbb{E}\int_{\{|u_n| \leq k\}} Z(u_n - u_m) |\nabla u_n|^p \,dx\,dt\,=0.
\end{align*}
Choosing $Z$ such that $Z(r)=1$ for $|r|\geq k'$ and $Z \geq 0$ on $\mathbb{R}$ such that $Z(0)=Z'(0)=0$, it follows
\begin{align*}
0 &\leq \lim\limits_{n,m \to \infty} \mathbb{E} J_{2,k,k'}^{n,m}\\
&= \lim\limits_{n,m \to \infty} \mathbb{E} \int_{\{|u_n| \leq k\} \cap \{|u_m| > k\} \cap \{|u_n - u_m| > k'\}} |\nabla u_n|^{p-2} \nabla u_n \cdot \nabla u_n \,dx\,dt \\
&\leq \lim\limits_{n,m \to \infty}  \mathbb{E} \int_{\{|u_n| \leq k\}} Z(u_n - u_m) |\nabla u_n|^p \,dx\,dt=0,
\end{align*}
which finally shows the validity of equality \eqref{22}.
Since equality \eqref{22} holds true, it follows that
\begin{align}\label{25}
\lim\limits_{n \to \infty} \mathbb{E} \int_0^T \int_D |\nabla T_k(u_n)|^{p-2} \nabla T_k(u_n) \cdot \nabla T_k(u_n) \,dx\,dt= \mathbb{E} \int_0^T \int_D \sigma_k \cdot \nabla T_k(u) \,dx\,dt.
\end{align}
Minty's trick yields $\sigma_k= |\nabla T_k(u)|^{p-2} \nabla T_k(u)$. We may conclude by using equality \eqref{25} that
\begin{align*}
\lim\limits_{n \to \infty} \Vert \nabla T_k(u_n) \Vert_{L^p(\Omega \times Q_T)^d}^p = \Vert \nabla T_k(u)\Vert_{L^p(\Omega \times Q_T)^d}^p.
\end{align*}
Since $L^p(\Omega \times Q_T)^d$ is uniformly convex and $\nabla T_k(u_n) \rightharpoonup \nabla T_k(u)$ in $L^p(\Omega \times Q_T)^d$ it yields
\begin{align*}
\nabla T_k(u_n) \to \nabla T_k(u) ~~~\textnormal{in} ~ L^p(\Omega \times Q_T)^d
\end{align*} 
which ends the proof of Lemma \ref{Lemma 6.3}. 

\end{proof}

\begin{lemma}\label{190206_lem01}
Let $(u_0^n)_n \subset L^2(\Omega \times D)$ be a sequence such that $u_0^n \to u_0$ in $L^1(\Omega \times D)$ and let $u_n$ be the unique strong solution to \eqref{1} with initial value $u_0^n$, i.e., \eqref{20} holds true. Then,
\begin{align}\label{190206_01}
\limsup_{k\rightarrow\infty}\limsup_{n\rightarrow\infty}\,\mathbb{E}\int_{\{k<|u_n|<k+1\}} |\nabla u_n|^p \,dx\,dt=0.
\end{align}
\end{lemma}
\begin{proof}
For fixed $l>0$, let $h_l:\mathbb{R}\rightarrow\mathbb{R}$ be defined by
\begin{align*}
h_l(r)=\begin{cases} 0 &,~|r| \geq l+1, \\ 
l+1-|r| &,~ l<|r|<l+1\\
1 &,~ |r|\leq l.
\end{cases}
\end{align*}
We plug $S(r)=\int_0^{r} h_l(\overline{r})(T_{k+1}(\overline{r})-T_k(\overline{r}))\,d\overline{r}$ and $\Psi\equiv 1$ in \eqref{190206_03} and take expectation to obtain
\begin{align}\label{190206_04}
I_1+I_2+I_3=I_4+I_5,
\end{align}
where
\begin{align*}
I_1&=\mathbb{E}\int_D \int_{u_0^n}^{u_n(t)}h_l(r)(T_{k+1}(r)-T_k(r))\, dr\,dx,\\
I_2&=\mathbb{E}\int_{\{l<|u_n|<l+1\}}-\operatorname{sign}(u_n)(T_{k+1}(u_n)-T_k(u_n))|\nabla u_n|^p\,dx\,ds,\\
I_3&=\mathbb{E}\int_{\{k<|u_n|<k+1\}}h_l(u_n)|\nabla u_n|^p\,dx\,ds,\\
I_4&=\frac{1}{2}\mathbb{E}\int_{\{l<|u_n|<l+1\}} -\operatorname{sign}(u_n)(T_{k+1}(u_n)-T_k(u_n))\Phi^2\,dx\,ds,\\
I_5&=\frac{1}{2}\mathbb{E}\int_{\{k<|u_n|<k+1\}}h_l(u_n)\Phi^2\,dx\,ds
\end{align*}
for all $t\in [0,T]$.
We can pass to the limit with $l\rightarrow\infty$ in \eqref{190206_04} by Lebesgue's Dominated Convergence theorem. We obtain
\begin{align}\label{190206_05}
J_1+J_2=J_3
\end{align}
where
\begin{align*}
J_1&=\mathbb{E}\int_D\int_{u_0^n}^{u_n(t)}T_{k+1}(r)-T_k(r)\, dr\,dx,\\
J_2&=\mathbb{E}\int_{\{k<|u_n|<k+1\}}|\nabla u_n|^p\,dx\,ds,\\
J_3&=\frac{1}{2} \mathbb{E}\int_{\{k<|u_n|<k+1\}}\Phi^2\,dx\,ds.
\end{align*}
Since $u_n\rightarrow u$ in $L^1(\Omega;C([0,T];L^1(D)))$ and $u_0^n\rightarrow u_0$ in $L^1(\Omega\times D)$, for $n\rightarrow \infty$, it follows that
\begin{align}\label{190206_08}
\lim_{k\rightarrow\infty}\lim_{n\rightarrow\infty}J_1=\lim_{k\rightarrow\infty} \int_D \int_{u_0}^{u(t)}T_{k+1}(r)-T_k(r)\, dr\,dx=0.
\end{align}
Now, the term $J_3$ desires our attention. For any $\sigma>0$ we have
\begin{align}\label{190206_06}
J_3&=\frac{1}{2}\,\mathbb{E}\int_0^t\int_{\{k<|u_n|<k+1\}}\left(\chi_{\{\Phi^2>\sigma\}}+\chi_{\{\Phi^2\leq\sigma\}}\right)\Phi^2\,dx\,ds\nonumber\\
&\leq\frac{1}{2}\,\mathbb{E}\int_0^t\int_{\{k<|u_n|<k+1\}}\sigma \,dx\,ds+\frac{1}{2}\,\mathbb{E}\int_0^t\int_{\{\Phi^2>\sigma\}}\Phi^2\,dx\,ds\nonumber\\
&\leq \frac{\sigma}{2k}\Vert u_n\Vert_{L^1(\Omega\times Q_T)} +\mathbb{E}\int_0^t\int_{\{\Phi^2>\sigma\}}\Phi^2\,dx\,ds
\end{align}
Thanks to the convergence of $(u_n)$, there exists a constant $C\geq 0$ not depending on the parameters $k$, $n$ and $\sigma$ such that 
\[\Vert u_n\Vert_{L^1(\Omega\times Q_T)}\leq C.\]
Thus, 
\begin{align}\label{190206_07}
\limsup_{k\rightarrow\infty}\limsup_{n\rightarrow\infty} J_3\leq \mathbb{E}\int_0^t\int_{\{\Phi^2>\sigma\}}\Phi^2\,dx\,ds
\end{align}
and therefore, passing to the limit with $\sigma\rightarrow\infty$, from \eqref{190206_07} and the nonnegativity of $J_3$ it follows that
\begin{align}\label{190206_09}
\limsup_{k\rightarrow\infty}\limsup_{n\rightarrow\infty} J_3=0.
\end{align}
Combining \eqref{190206_05}, \eqref{190206_08} and \eqref{190206_09}, and using the nonnegativity of $J_2$, we arrive at \eqref{190206_01}.
\end{proof}
We have 
\[\chi_{\{k<|u_n|<k+1\}}\chi_{\{|u|\neq k\}}\chi_{\{|u|\neq k+1\}}\rightarrow \chi_{\{k<|u|<k+1\}}\chi_{\{|u|\neq k\}}\chi_{\{|u|\neq k+1\}}\]
for $n\rightarrow\infty$ in $L^r(\Omega\times Q_T)$ for any $1\leq r<\infty$ and a.e. in $\Omega\times Q_T$. From Lemma \ref{Lemma 6.3} we recall that for any $k>0$,
\begin{align*}
\nabla T_k(u_n) \to \nabla T_k(u) ~~~\textnormal{in} ~ L^p(\Omega \times Q_T)^d
\end{align*}
for $n\rightarrow\infty$, thus, passing to a not relabeled subsequence if necessary, also a.s. in $\Omega\times Q_T$. Since $\nabla T_k(u)=0$ a.s. on $\{|u|=m\}$ for any $m\geq 0$, from Fatou's Lemma it follows that 
\begin{align}\label{190206_02}
&\liminf_{n\rightarrow\infty}\, \mathbb{E}\int_{\{k<|u_n|<k+1\}} |\nabla u_n|^p \,dx\,dt\nonumber\\
&\geq \liminf_{n\rightarrow\infty} \mathbb{E}\,\int_{\{k<|u_n|<k+1\}} |\nabla u_n|^p\chi_{\{|u|\neq k\}}\chi_{\{|u|\neq k+1\}} \,dx\,dt\nonumber\\
&\geq\mathbb{E}\int_{\{k<|u|<k+1\}} |\nabla u|^p \chi_{\{|u|\neq k\}}\chi_{\{|u|\neq k+1\}} \,dx\,dt\nonumber\\
&=\mathbb{E}\int_{\{k<|u|<k+1\}} |\nabla u|^p \,dx\,dt
\end{align}
and the energy dissipation condition $(iii)$ follows combining \eqref{190206_01} with \eqref{190206_02}.
\section{Uniqueness of renormalized solutions}
\begin{theorem}\label{Theorem 7.1}
Let $u,v$ be renormalized solutions to \eqref{1} with initial data $u_0 \in L^1(\Omega \times D)$ and $v_0 \in L^1(\Omega \times D)$, respectively. Then we get
\begin{align}\label{34}
\int_D |u(t) - v(t)| \,dx \leq \int_D |u_0 - v_0| \,dx
\end{align}
a.s. in $\Omega$, for all $t \in [0,T]$.
\end{theorem}
\begin{proof}
This proof is inspired by the uniqueness proof in \cite{DBFMHR}. We know that $u$ satisfies the SPDE
\begin{align}\label{35}
&dS(u)-\operatorname{div}\,(S'(u)|\nabla u|^{p-2}\nabla u) \,dt+S''(u)|\nabla u|^p\,dt\nonumber\\
&=\Phi S'(u) \,d\beta +\frac{1}{2}S''(u)\Phi^2 \,dt
\end{align}
for all $S \in C^2(\mathbb{R})$ such that $\textnormal{supp}~ S'$ compact. Moreover, $v$ satisfies an analogous SPDE. Subtracting both equalities yields
\begin{align}\label{36}
&S(u(t)) - S(v(t)) = S(u_0) - S(v_0) + \int_0^t \textnormal{div}[S'(u) |\nabla u|^{p-2} \nabla u - S'(v)|\nabla v|^{p-2} \nabla v]\,ds  \notag \\
- &\int_0^t \left(S''(u)|\nabla u|^p  - S''(v)|\nabla v|^p\right)\,ds + \int_0^t \Phi (S'(u) - S'(v))  \,d\beta \\
+ &\frac{1}{2} \int_0^t \Phi^2 (S''(u) - S''(v)) \,ds \notag
\end{align}
in $W^{-1,p'}(D) + L^1(D)$ for all $t \in [0,T]$, a.s. in $\Omega$.\\
\newline
Now we set $S(r):= T_s^{\sigma}(r)$ for $r \in \mathbb{R}$ and $s,\sigma >0$ and define $T_s^{\sigma}$ as follows: Firstly, we define for all $r \in \mathbb{R}$
\begin{align*}
(T_s^{\sigma})'(r)= \begin{cases}
1 ,~&\textnormal{if}~|r| \leq s, \\
\frac{1}{\sigma}(s+\sigma - |r|), ~&\textnormal{if}~ s < |r| < s + \sigma, \\
0, ~&\textnormal{if}~ |r| \geq s+ \sigma.
\end{cases}
\end{align*}
Then we set $T_s^{\sigma}(r):= \int_0^r (T_s^{\sigma})'(\tau) \,d\tau$. Furthermore we have the weak derivative
\begin{align*}
(T_s^{\sigma})''(r)= \begin{cases}
-\frac{1}{\sigma} \textnormal{sgn}(r), ~&\textnormal{if}~ s < |r| < s + \sigma, \\
0, ~&\textnormal{otherwise}.
\end{cases}
\end{align*}
Applying Proposition \ref{Proposition 4.2} to equality \eqref{36} with $S(r)= \frac{1}{k} \tilde{T}_k(r)=\frac{1}{k}\int_0^r T_k(\overline{r})\,d\overline{r}$ (see also Remark \ref{Remark 5.2}) yields
\begin{align*}
&\int_D \bigg( \frac{1}{k} \tilde{T}_k( T_s^{\sigma}(u(t)) - T_s^{\sigma}(v(t))) - \frac{1}{k} \tilde{T}_k( T_s^{\sigma}(u_0) - T_s^{\sigma}(v_0)) \bigg) \psi \,dx \\
- &\int_0^t \langle \textnormal{div} ((T_s^{\sigma})'(u) |\nabla u|^{p-2} \nabla u - (T_s^{\sigma})'(v) |\nabla v|^{p-2} \nabla v), \frac{1}{k} T_k( T_s^{\sigma}(u) - T_s^{\sigma}(v))\psi \rangle \,dr \\
= &\int_D \int_0^t \bigg( - ((T_s^{\sigma})''(u) |\nabla u|^p - (T_s^{\sigma})''(v) |\nabla v|^p) + \frac{1}{2} \Phi^2 ((T_s^{\sigma})''(u) -(T_s^{\sigma})''(v)) \bigg) \cdot \\
\cdot &\frac{1}{k}T_k (T_s^{\sigma}(u) - T_s^{\sigma}(v)) \psi \,dr \,dx \\
+ &\int_D \int_0^t \Phi ((T_s^{\sigma})'(u) - (T_s^{\sigma})'(v)) \frac{1}{k} T_k (T_s^{\sigma}(u) - T_s^{\sigma}(v)) \psi \,d\beta \,dx \\
+ &\frac{1}{2} \int_D \int_0^t \Phi^2 (T_s^{\sigma}(u) - T_s^{\sigma}(v))^2 \frac{1}{k} \chi_{\{|T_s^{\sigma}(u) - T_s^{\sigma}(v)|<k\}} \psi \,dr\,dx
\end{align*}
for all $s,\sigma, k>0$ and all $\psi \in C^{\infty}([0,T] \times \overline{D})$. Using $\psi=1$ yields
\begin{align}\label{37}
&\int_D \bigg( \frac{1}{k} \tilde{T}_k( T_s^{\sigma}(u(t)) - T_s^{\sigma}(v(t))) - \frac{1}{k} \tilde{T}_k( T_s^{\sigma}(u_0) - T_s^{\sigma}(v_0)) \bigg) \,dx \notag \\
- &\int_0^t \langle \textnormal{div} ((T_s^{\sigma})'(u) |\nabla u|^{p-2} \nabla u - (T_s^{\sigma})'(v) |\nabla v|^{p-2} \nabla v), \frac{1}{k} T_k( T_s^{\sigma}(u) - T_s^{\sigma}(v)) \rangle \,dr \notag \\
= &\int_D \int_0^t \bigg( - ((T_s^{\sigma})''(u) |\nabla u|^p - (T_s^{\sigma})''(v) |\nabla v|^p) + \frac{1}{2} \Phi^2 ((T_s^{\sigma})''(u) -(T_s^{\sigma})''(v)) \bigg) \cdot \\
\cdot &\frac{1}{k}T_k (T_s^{\sigma}(u) - T_s^{\sigma}(v)) \,dr \,dx \notag \\
+ &\int_D \int_0^t \Phi ((T_s^{\sigma})'(u) - (T_s^{\sigma})'(v)) \frac{1}{k} T_k (T_s^{\sigma}(u) - T_s^{\sigma}(v)) \,d\beta\,dx \notag\\
+ &\frac{1}{2} \int_D \int_0^t \Phi^2 (T_s^{\sigma}(u) - T_s^{\sigma}(v))^2 \frac{1}{k}  \chi_{\{|T_s^{\sigma}(u) - T_s^{\sigma}(v)|<k\}} \,dr\,dx. \notag
\end{align}
a.s. in $\Omega$ for any $t\in [0,T]$. We write equality \eqref{37} as
\begin{align*}
I_1 + I_2= I_3+ I_4 + I_5.
\end{align*}
For $\omega\in \Omega$ and $t\in [0,T]$ fixed, we pass to the limit with $\sigma \to 0$ firstly, then we pass to the limit $k \to 0$ and finally we let $s \to \infty$. Before we do so, we have to give some remarks on $T_s^{\sigma}$.\\
By definition of $(T_s^{\sigma})'$ we see immediately that $(T_s^{\sigma})'(r) \to \chi_{\{|r|\leq s\}}$ pointwise for all $r \in \mathbb{R}$ as $\sigma \to 0$. Since $|(T_s^{\sigma})'| \leq 1$ on $\mathbb{R}$ we have,
\begin{align*}
(T_s^{\sigma})'(u) \to \chi_{\{|u|\leq s\}}
\end{align*}
in $L^1(Q_T)$ a.s. in $\Omega$ and a.e. in $\Omega \times Q_T$ 
as $\sigma \to 0$. An analogous result holds true for $v$ instead of $u$.\\
For $0<\sigma<1$ and fixed $s>0$, we have $\textnormal{supp}(T_s^{\sigma})' \subset [-s-1,s+1]$. Therefore $T_s^{\sigma}$ is bounded in $L^{\infty}(\mathbb{R})$ for fixed $s$ and we may conclude $ T_s^{\sigma}(u) \to T_s(u)$ a.e. in $\Omega \times Q_T$ and in $L^1(Q_T)$ a.s. in $\Omega$ as $\sigma \to 0$. Furthermore, we have $\nabla T_s^{\sigma}(u)= \nabla T_s^{\sigma}(T_{s+1}(u))= \nabla T_{s+1}(u) (T_s^{\sigma})'(u)$. Since $|(T_s^{\sigma})'| \leq 1$ on $\mathbb{R}$, we have
\begin{align*}
\nabla T_{s+1}(u) (T_s^{\sigma})'(u) \to \nabla T_{s+1}(u) \chi_{\{|u| \leq s\}} ~~~\textnormal{in}~L^p(Q_T)^d
\end{align*}
for $\sigma\rightarrow 0$ a.s. in $\Omega$. Since $\nabla T_{s+1}(u) \chi_{\{|u| \leq s\}}= \nabla T_s(u)$ we have
\begin{align}\label{38}
\nabla T_s^{\sigma}(u) \to \nabla T_s(u) ~~~\textnormal{in} ~L^p(Q_T)^d
\end{align}
for $\sigma\rightarrow 0$ a.s. in $\Omega$. \\
Let us consider $I_1$. By Lebesgue's Theorem it follows
\begin{align*}
\lim\limits_{\sigma \to 0} I_1= \int_D \bigg( \frac{1}{k} \tilde{T}_k( T_s(u(t)) - T_s(v(t))) - \frac{1}{k} \tilde{T}_k( T_s(u_0) - T_s(v_0)) \bigg) \,dx
\end{align*}
a.s. in $\Omega$, for all $t \in [0,T]$. Since $\frac{1}{k} \tilde{T}_k \to | \cdot |$ pointwise in $\mathbb{R}$ as $k \to 0$, we may conclude
\begin{align*}
\lim\limits_{k \to 0} \lim\limits_{\sigma \to 0} I_1 = \int_D |T_s(u(t)) - T_s(v(t))| - |T_s(u_0) - T_s(v_0)| \,dx
\end{align*}
a.s. in $\Omega$, for all $t \in [0,T]$. Finally, since $u(t),v(t),u_0,v_0 \in L^1(D)$ for all $t \in [0,T]$, a.s. in $\Omega$ again Lebesgue's Theorem yields
\begin{align}\label{39}
\lim\limits_{s \to \infty}\lim\limits_{k \to 0} \lim\limits_{\sigma \to 0} I_1 = \int_D |u(t) - v(t)| - |u_0 - v_0| \,dx
\end{align}
a.s. in $\Omega$, for all $t \in [0,T]$.\\
\newline
Next we want to show that
\begin{align*}
\liminf\limits_{\sigma \to 0} I_2 \geq 0
\end{align*}
a.s. in $\Omega$. \\
Since $\textnormal{supp}(T_s^{\sigma})' \subset [-s-1,s+1]$ for $0< \sigma < 1$ we can estimate
\begin{align*}
I_2 = &\int_D \int_0^t ((T_s^{\sigma})'(u) |\nabla u|^{p-2} \nabla u - (T_s^{\sigma})'(v) |\nabla v|^{p-2} \nabla v) \frac{1}{k} \nabla T_k( T_s^{\sigma}(u) - T_s^{\sigma}(v)) \,dr \,dx \\
= &\int_D \int_0^t \bigg((T_s^{\sigma})'(u) |\nabla T_{s+1}(u)|^{p-2} \nabla T_{s+1}(u) - (T_s^{\sigma})'(v) |\nabla T_{s+1}(v)|^{p-2} \nabla T_{s+1}(v) \bigg) \cdot \\
\cdot &\frac{1}{k} \nabla T_k( T_s^{\sigma}(u) - T_s^{\sigma}(v)) \,dr \,dx.
\end{align*}
Hence
\begin{align}\label{40}
\lim\limits_{\sigma \to 0} I_2 &= \frac{1}{k} \int_{\{|T_s(u) - T_s(v)|<k\}} \bigg(|\nabla T_{s}(u)|^{p-2} \nabla T_{s}(u) - |\nabla T_{s}(v)|^{p-2} \nabla T_{s}(v) \bigg) \cdot \\
&\cdot \nabla ( T_s(u) - T_s(v)) \,dr \,dx \geq 0 \notag
\end{align}
a.s. in $\Omega$, for all $t \in [0,T]$.\\
\newline
Considering $I_5$ we can easily estimate
\begin{align*}
0 \leq I_5 \leq \frac{1}{2} \int_D \int_0^t \Phi^2 \cdot k \,dr \, dx
\end{align*}
a.s. in $\Omega$. Thus we get
\begin{align}\label{41}
\limsup\limits_{k \to 0} \limsup\limits_{\sigma \to 0} I_5=0
\end{align}
a.s. in $\Omega$.\\
We write
\begin{align*}
I_3=I_3^1 + I_3^2,
\end{align*}
where
\begin{align*}
I_3^1 &= - \int_D \int_0^t ((T_s^{\sigma})''(u) |\nabla u|^p - (T_s^{\sigma})''(v) |\nabla v|^p) \frac{1}{k}T_k (T_s^{\sigma}(u) - T_s^{\sigma}(v)) \,dr \,dx,\\
I_3^2 &= \int_D \int_0^t \frac{1}{2} \Phi^2 ((T_s^{\sigma})''(u) -(T_s^{\sigma})''(v)) \frac{1}{k}T_k (T_s^{\sigma}(u) - T_s^{\sigma}(v)) \,dr \,dx.
\end{align*}
For $0< \sigma < 1$ we have
\begin{align*}
|I_3^1| \leq \frac{1}{\sigma} \int_{\{s < |u| < s+\sigma \}} |\nabla u|^p \,dr\,dx+ \frac{1}{\sigma}\int_{\{s < |v| < s+\sigma \}} |\nabla v|^p\,dr\,dx
\end{align*}
and we want show that there is a sequence $(s_j)_{j \in \mathbb{N}} \subset \mathbb{N}$ such that $s_j \to \infty$ as $j \to \infty$ and
\begin{align*}
\lim\limits_{j \to \infty} \limsup\limits_{\sigma \to 0}\frac{1}{\sigma} \int_{\{s_j < |u| < s_j+\sigma \}} |\nabla u|^p \,dr\,dx+ \frac{1}{\sigma}\int_{\{s_j < |v| < s_j+\sigma \}} |\nabla v|^p\,dr\,dx =0.
\end{align*}
According to Lemma 6 in \cite{DBFMHR} it is sufficient to show that for any $s \in \mathbb{N}$ there exists a nonnegative function $F \in L^1(\Omega\times (0,t) \times D)$ such that for fixed $\omega \in \Omega$
\begin{align*}
&\frac{1}{\sigma} \int_{\{s < |u| < s+\sigma \}} |\nabla u|^p \,dr\,dx+ \frac{1}{\sigma}\int_{\{s < |v| < s+\sigma \}} |\nabla v|^p\,dr\,dx \\
&\leq \frac{1}{\sigma}\left(\int_{\{s \leq |u| \leq s+\sigma \}} F \,dr \,dx +\int_{\{s \leq |v| \leq s+\sigma \}} F \,dr \,dx\right) +\epsilon(\sigma,s,\omega),
\end{align*}
where $\limsup\limits_{\sigma \to 0} \epsilon(\sigma,s,\omega) \to 0$ as $s \to \infty$ to conclude that 
\begin{align*}
\lim\limits_{j \to \infty} \limsup\limits_{\sigma \to 0} |I_3^1|=0,
\end{align*}
where $s$ is exchanged by $s_j$ in $I_3^1$.\\
For symmetry reasons, we only have to show the existence of a nonnegative function $F \in L^1(\Omega \times (0,t) \times D)$ such that for fixed $\omega \in \Omega$
\begin{align*}
\frac{1}{\sigma} \int_{\{s < |u| < s+\sigma \}} |\nabla u|^p \,dr\,dx \leq \frac{1}{\sigma} \int_{\{s \leq |u| \leq s+\sigma \}} F \,dr \,dx + \epsilon(\sigma,s,\omega).
\end{align*}
To this end, we plug $S(u(t))=\frac{1}{\sigma} \int_0^{u(t)} h_l(\tau)(T_{s+ \sigma}(\tau)-T_s(\tau))\,d\tau$ and $\Psi\equiv 1$ in the renormalized formulation for $u$ to obtain
\begin{align*}
L_1+L_2+L_3=L_4+L_5+L_6
\end{align*}
a.s. in $\Omega$, where
\begin{align*}
L_1 &:= \frac{1}{\sigma} \int_D \int_{u_0}^{u(t)} h_l(\tau)(T_{s+ \sigma}(\tau) - T_s(\tau))  \, d\tau\, dx , \\
L_2 &:= \frac{1}{\sigma} \int_D \int_0^t - \operatorname{sign}(u) \chi_{\{l < |u| < l+1\}} (T_{s+ \sigma}(u) - T_s(u)) |\nabla u|^p \, dr \, dx, \\
L_3 &:= \frac{1}{\sigma} \int_D \int_0^t h_l(u) \chi_{\{s < |u| < s + \sigma\}} |\nabla u|^p \, dr\, dx,  \\
L_4 &:= \frac{1}{\sigma}\int_0^t\int_D  h_l(u) (T_{s+ \sigma}(u) - T_s(u)) \Phi \, dx \, d\beta, \\
L_5 &:= \frac{1}{2\sigma} \int_D \int_0^t - \operatorname{sign}(u) \chi_{\{l < |u| < l+1\}} (T_{s+ \sigma}(u) - T_s(u)) \Phi^2  \, dr\, dx, \\
L_6 &:= \frac{1}{2\sigma} \int_D \int_0^t h_l(u) \chi_{\{s < |u| < s + \sigma\}} \Phi^2  \, dr\, dx.
\end{align*}
It is straightforward to pass to the limit with $l \to \infty$ for a.e. $\omega\in\Omega$ in $L_1$, $L_3$, $L_4$, $L_5$ and $L_6$. We have
\begin{align*}
|L_2|\leq \frac{\sigma}{\sigma}\int_{\{l<|u|<l+1\}}|\nabla u|^p\,dr\,dx
\end{align*}
a.s. in $\Omega$. In order to pass to the limit with $l\rightarrow\infty$ in $L_2$, we recall that from the energy dissipation condition $(iii)$ it follows that, passing to a not relabeled subsequence if necessary, 
\[\lim_{l\rightarrow\infty}\int_{\{l<|u|<l+1\}}|\nabla u|^p\,dr\,dx=0\]
a.s. in $\Omega$ and therefore $\lim\limits_{l \to \infty} L_2=0$ a.s. in $\Omega$. After this passage to the limit the remaining terms are
\begin{align*}
J_1 + J_2 = J_3 + J_4
\end{align*}
a.s. in $\Omega$, where
\begin{align*}
J_1 &:= \frac{1}{\sigma} \int_D \int_{u_0}^{u(t)} (T_{s+ \sigma}(\tau) - T_s(\tau)) \, d\tau\, dx, \\
J_2 &:= \frac{1}{\sigma} \int_{\{s < |u| < s + \sigma\}} |\nabla u|^p \, dr\, dx,  \\
J_3 &:= \frac{1}{\sigma} \int_0^t \int_D (T_{s+ \sigma}(u) - T_s(u)) \Phi \, dx \, d\beta,  \\
J_4 &:= \frac{1}{\sigma} \int_{\{s < |u| < s + \sigma\}} \frac{1}{2} \Phi^2 \, dr \, dx .
\end{align*}
We set $F=: \frac{1}{2} \Phi^2$.\\
It is left to show that $\lim\limits_{s \to \infty}\limsup\limits_{\sigma\to 0} \epsilon(\sigma, s,\omega)=0$, where
\begin{align*}
\epsilon(\sigma,s,\omega)= \frac{1}{\sigma} \int_D \int^{u(t)}_{u_0} (T_{s+ \sigma}(\tau) - T_s(\tau))  \, d\tau \, dx  + \frac{1}{\sigma} \int_0^t \int_D (T_{s+ \sigma}(u) - T_s(u)) \Phi \, dx \, d\beta.
\end{align*}
Since, $|\frac{1}{\sigma} (T_{s+ \sigma}(\tau) - T_s(\tau))| \leq 1$ for $0<\sigma\leq 1$, it follows that
\begin{align*}
\left|\int_{u_0}^{u(t)}\frac{1}{\sigma}(T_{s+ \sigma}(\tau) - T_s(\tau))\,d\tau\right|\leq |u(t)-u_0|\in L^1(\Omega\times D).
\end{align*}
Moreover,
\begin{align*}
\frac{1}{\sigma} (T_{s+ \sigma}(\tau) - T_s(\tau)) \to \operatorname{sign}(\tau) \chi_{\{|\tau| \geq s \}}
\end{align*}
a.e. in $\{|\tau| \neq s\}$ and from Lebesgue's dominated convergence theorem it follows that
\begin{align*}
&\lim\limits_{\sigma \to 0}\int_D \int^{u(t)}_{u_0} \frac{1}{\sigma} (T_{s+ \sigma}(\tau) - T_s(\tau)) \, d\tau \, dx\\
&=\lim\limits_{\sigma \to 0}\int_D \int^{u(t)}_{u_0} \frac{1}{\sigma} (T_{s+ \sigma}(\tau) - T_s(\tau))\chi_{\{|\tau| \neq s\}} \, d\tau \, dx  \\
&=\int_D \int^{u(t)}_{u_0} \operatorname{sign}(\tau) \chi_{\{|\tau| \geq s \}}\chi_{\{|\tau| \neq s\}} \, d\tau \, dx \to 0
\end{align*}
as $s \to \infty$ a.s. in $\Omega$. Moreover, since $\frac{1}{\sigma^2}(T_{s+ \sigma}(u) - T_s(u))^2 \Phi^2 \to \chi_{\{|u| \geq s \}} \Phi^2$ a.e. in $\{ |u| \neq s \}$ and $\frac{1}{\sigma^2}(T_{s+ \sigma}(u) - T_s(u))^2 \Phi^2 \leq \Phi^2\in L^2(\Omega\times Q_T)$ we obtain
\begin{align*}
&\lim\limits_{\sigma \to 0} \mathbb{E} \,\bigg | \int_D \int_0^t \frac{1}{\sigma} (T_{s+ \sigma}(u) - T_s(u)) \Phi \, dx \, d\beta  \bigg |^2 \\
&=\lim\limits_{\sigma \to 0} \mathbb{E} \,\bigg |\int_0^t  \int_D\frac{1}{\sigma} (T_{s+ \sigma}(u) - T_s(u)) \Phi\chi_{\{|u| \neq s\}} \, dx \, d\beta  \bigg |^2 \\
&=\lim\limits_{\sigma \to 0} \mathbb{E} \,\int_0^t \int_D \frac{1}{\sigma^2} (T_{s+ \sigma}(u) - T_s(u))^2 \Phi^2\chi_{\{|u| \neq s\}} \, dx \, dr \\
&=\mathbb{E}\,\int_0^t \int_D \chi_{\{|u| \geq s \}} \Phi^2\chi_{\{|u| \neq s\}}  \, dx \, dr \to 0
\end{align*}
as $s \to \infty$. Passing to suitable subsequences in $\sigma$ and $s$ we may conclude
\begin{align*}
\lim\limits_{s \to \infty}  \lim\limits_{\sigma \to 0} \int_0^t \int_D \frac{1}{\sigma} (T_{s+ \sigma}(u) - T_s(u)) \Phi \, dx \, d\beta=0
\end{align*}
a.s. in $\Omega$.
It follows that
\begin{align*}
\lim\limits_{s \to \infty} \lim\limits_{\sigma \to 0} \epsilon(\sigma,s,\omega) =0.
\end{align*}
Therefore we get
\begin{align}\label{42}
\lim\limits_{j \to \infty} \limsup\limits_{\sigma \to 0} |I_3^1|=0.
\end{align}
Now let us consider the integrand of $I_3^2$ pointwise in $Q_t$ for a fixed $w \in \Omega$. We have $(T_s^{\sigma})''(u) \to 0$ a.e. in $Q_T$ as $\sigma \to 0$. Hence the whole integrand of $I_3^2$ tends to $0$ a.e. in $Q_t$ as $\sigma \to 0$. W.l.o.g. assume that $u \geq v$ at some point in $Q_t$. Then we have $\frac{1}{k}T_k (T_s^{\sigma}(u) - T_s^{\sigma}(v)) \geq 0$. Furthermore, we have $(T_s^{\sigma})''(u) -(T_s^{\sigma})''(v)>0$ if and only if $s<v<s+\sigma <u$ or $v<-s-\sigma<u<-s$. In both cases we can estimate easily that $ T_s^{\sigma}(u) - T_s^{\sigma}(v) \leq \frac{\sigma}{2}$. Since $|(T_s^{\sigma})''(u) -(T_s^{\sigma})''(v)| \leq \frac{2}{\sigma}$ the results above yield
\begin{align*}
\frac{1}{2} \Phi^2 ((T_s^{\sigma})''(u) -(T_s^{\sigma})''(v)) \frac{1}{k}T_k (T_s^{\sigma}(u) - T_s^{\sigma}(v)) \leq \frac{1}{2k} \Phi^2
\end{align*}
a.e. in $Q_t$ and $\frac{1}{2k} \Phi^2 \in L^1(Q_t)$. Therefore Fatou's Lemma yields
\begin{align}\label{43}
\limsup\limits_{\sigma \to 0} I_3^2 \leq 0
\end{align}
a.s. in $\Omega$.\\
Now we consider the term $I_4$. Applying the It\^{o} isometry yields
\begin{align*}
&\mathbb{E} \bigg| \int_D \int_0^t \Phi \cdot ((T_s^{\sigma})'(u) - (T_s^{\sigma})'(v)) \frac{1}{k} T_k (T_s^{\sigma}(u) - T_s^{\sigma}(v)) \,d\beta \,dx \bigg|^2 \\
= &\mathbb{E} \int_D \int_0^t |\Phi ((T_s^{\sigma})'(u) - (T_s^{\sigma})'(v)) \frac{1}{k} T_k (T_s^{\sigma}(u) - T_s^{\sigma}(v))|^2 \,dr \,dx
\end{align*}
for all $t \in [0,T]$. As $|(T_s^{\sigma})'| \leq 1 $ and $|\frac{1}{k}T_k| \leq 1$ we may conclude by applying the Lebesgue's Theorem
\begin{align*}
\lim\limits_{s \to \infty} \lim\limits_{k \to 0} \lim\limits_{\sigma \to 0} &\mathbb{E} \int_D \int_0^t |\Phi \cdot ((T_s^{\sigma})'(u) - (T_s^{\sigma})'(v)) \frac{1}{k} T_k (T_s^{\sigma}(u) - T_s^{\sigma}(v))|^2 \,dr \,dx \\
= \lim\limits_{s \to \infty} \lim\limits_{k \to 0} &\mathbb{E} \int_D \int_0^t |\Phi \cdot (\chi_{\{|u| \leq s\}} - \chi_{\{|v| \leq s\}}) \frac{1}{k} T_k (T_s(u) - T_s(v))|^2 \,dr \,dx \\
= \lim\limits_{s \to \infty} &\mathbb{E} \int_D \int_0^t |\Phi \cdot(\chi_{\{|u| \leq s\}} - \chi_{\{|v| \leq s\}}) \operatorname{sign}_0(T_s(u) - T_s(v))|^2 \,dr \,dx \\
= &0
\end{align*}
for all $t \in [0,T]$. Therefore we obtain
\begin{align*}
\lim\limits_{s \to \infty} \lim\limits_{k \to 0} \lim\limits_{\sigma \to 0} \int_D \int_0^t \Phi \cdot ((T_s^{\sigma})'(u) - (T_s^{\sigma})'(v)) \frac{1}{k} T_k (T_s^{\sigma}(u) - T_s^{\sigma}(v)) \,d\beta \,dx =0 ~~~\textnormal{in}~L^2(\Omega).
\end{align*}
Passing to suitable subsequences in $\sigma,k$ and $s$ it follows that 
\begin{align}\label{44}
\lim\limits_{s \to \infty} \lim\limits_{k \to 0} \lim\limits_{\sigma \to 0} \int_D \int_0^t \Phi \cdot ((T_s^{\sigma})'(u) - (T_s^{\sigma})'(v)) \frac{1}{k} T_k (T_s^{\sigma}(u) - T_s^{\sigma}(v)) \,d\beta \,dx =0
\end{align}
a.s. in $\Omega$, for all $t \in [0,T]$. \\
From \eqref{39} - \eqref{44} it follows
\begin{align*}
\int_D |u(t) - v(t)| \,dx \leq \int_D |u_0 - v_0| \,dx
\end{align*}
a.s. in $\Omega$, for all $t \in [0,T]$.
\end{proof}
\section{Markov property}
Note that it is possible to replace the starting time $0$ by a starting time $r \in [0,T]$. In this case, we consider the filtration starting at time $r$, i.e., $(\mathcal{F}_t)_{t \in [r,T]}$. Then, $\tilde{\beta}_t:= \beta_t - \beta_r$, $t \in [r,T]$, is a Brownian motion with respect to $(\mathcal{F}_t)_{t\in [r,T]}$ such that $\sigma(\tilde{\beta_t}, ~t\geq r)$ is independent of $\mathcal{F}_r$ (see, e.g., Remark 3.2. in \cite{PB}). Moreover, the augmentation $\tilde{\mathcal{F}_t}$ of $\sigma(\tilde{\beta_t}, ~t\geq r)$ is right-continuous and independent of $\mathcal{F}_r$. Furthermore, we have $d\beta_t= d\tilde{\beta}_t$ and all results and arguments still hold true in the case of a starting time $r \in [0,T]$ and $\mathcal{F}_r$-measurable initial conditions $u_r \in L^1(\Omega \times D)$.
In this section, we denote by $u(t,r,u_r)$ the unique renormalized solution of \eqref{1} starting in $u_r$ at time $r$ for $t,r \in [0,T]$ with $r \leq t$ and $u_r \in L^1(\Omega \times D)$ $\mathcal{F}_r$-measurable.
\begin{proposition}\label{Proposition 8.1}
For all $r,s,t \in [0,T]$ with $r \leq s \leq t$ and all $u_r \in L^1(\Omega \times D)$ $\mathcal{F}_r$-measurable we have
\begin{align*}
u(t,s, u(s,r,u_r))= u(t,r,u_r)
\end{align*}
a.s. in $\Omega$.
\end{proposition}
\begin{proof}
The proof is similar to the proof in \cite{PB} or \cite{WLMR}. Let $r \in [0,T]$, $\psi \in \mathcal{C}^{\infty}([r,T] \times \overline{D})$ and $S \in \mathcal{C}^2(\mathbb{R})$ such that $S'$ has compact support with $S'(0)=0$ or $\psi(t,x)=0$ for all $(t,x) \in [r,T] \times \partial D$. Now we fix $s,t \in [0,T]$ with $r \leq s \leq t$ and $u_r \in L^1(\Omega \times D)$ $\mathcal{F}_r$-measurable. Since $ u(\cdot,r,u_r)$ is the unique renormalized solution to \eqref{1} starting in $u_r$ at time $r$ we have
\begin{align*}
&\int_D S(u(t,r,u_r))\psi(t) \,dx \\
= &\int_D S(u_r) \psi(r) \,dx - \int_r^t \int_D S''(u(\tau,r,u_r)) |\nabla u(\tau,r,u_r))|^p \psi \,dx\,d\tau \\
- &\int_r^t \int_D S'(u(\tau,r,u_r)) |\nabla u(\tau,r,u_r)|^{p-2} \nabla u(\tau,r,u_r) \cdot \nabla \psi \,dx\,d\tau\\
+ &\int_r^t \int_D S'(u(\tau,r,u_r)) \psi \Phi \,dx\,d\beta + \int_r^t \int_D S(u(\tau,r,u_r)) \psi_t \,dx\,d\tau 
\\
+ \frac{1}{2} &\int_r^t \int_D  S''(u(\tau,r,u_r)) \psi \Phi^2 \,dx\,d\tau 
\end{align*}
\begin{align*}
= &\int_D S(u(s,r,u_r))\psi(s) \,dx - \int_s^t \int_D S''(u(\tau,r,u_r)) |\nabla u(\tau,r,u_r))|^p \psi \,dx\,d\tau \\
- &\int_s^t \int_D S'(u(\tau,r,u_r)) |\nabla u(\tau,r,u_r)|^{p-2} \nabla u(\tau,r,u_r) \cdot \nabla \psi \,dx\,d\tau\\
+ &\int_s^t \int_D S'(u(\tau,r,u_r)) \psi \Phi \,dx\,d\beta + \int_s^t \int_D S(u(\tau,r,u_r)) \psi_t \,dx\,d\tau 
\\
+ \frac{1}{2} &\int_s^t \int_D  S''(u(\tau,r,u_r)) \psi \Phi^2 \,dx\,d\tau 
\end{align*}
a.s. in $\Omega$. Therefore $u(t,r,u_r)$ is a renormalized solution to \eqref{1} starting in $u(s,r,u_r)$ at time $s$. Uniqueness yields the result.
\end{proof}
\begin{theorem}\label{Theorem 8.2}
Let $u_r \in L^1(\Omega \times D)$, $r \in [0,T]$ be $\mathcal{F}_r$-measurable. The unique renormalized solution $u(t)=u(t,r,u_r)$, $t \in [r,T]$, of \eqref{1} starting in $u_r$ at time $r$ satisfies the Markov property in the following sense:\\
For every bounded and $\mathcal{B}(L^1(D))$-measurable function $G:L^1(D) \to \mathbb{R}$ and all $s,t \in [r,T]$ with $s \leq t$ we have
\begin{align*}
\mathbb{E} [G(u(t))| \mathcal{F}_s](\omega)= \mathbb{E} [G(u(t,s,u(s,r,u_r)(\omega)))]
\end{align*}
for a.e. $\omega \in \Omega$.
\end{theorem}
\begin{proof}
We apply Lemma 4.1 in \cite{PB} (The freezing Lemma). To this end we set for fixed $r,t,s \in [0,T]$ with $r \leq s \leq t$, $u_r \in L^1(\Omega \times D)$ $\mathcal{F}_r$-measurable and a fixed bounded and $\mathcal{B}(L^1(D))$-measurable function $G:L^1(D) \to \mathbb{R}$: $\mathcal{D}= \mathcal{F}_s$, $\mathcal{G}=\tilde{\mathcal{F}}_t$, $E=L^1(D)$, $\mathcal{E}=\mathcal{B}(L^1(D))$, $X=u(s)=u(s,r,u_r)$ and $\psi: L^1(D) \times \Omega \to \mathbb{R}$, $\psi(x,\omega)= G(u(t,s,x)(\omega))$.\\
It is only left to prove that $\psi$ is $\mathcal{B}(L^1(D)) \otimes \tilde{\mathcal{F}}_t$-measurable. Since $G$ is $\mathcal{B}(L^1(D))$ measurable it is left to show that $\phi: L^1(D) \times \Omega \to L^1(D)$, $\phi(x,\omega)= u(t,s,x)(\omega)$ is $\mathcal{B}(L^1(D)) \otimes \tilde{\mathcal{F}}_t-\mathcal{B}(L^1(D))$-measurable. To this end we show that $\phi$ is Carath\'{e}odory, i.e.,
\begin{align*}
&(i)~ \Omega \ni \omega \mapsto \phi(x,\omega) ~\textnormal{is}~ \tilde{\mathcal{F}}_t-\textnormal{mesaurable for all}~x \in L^1(D), \\
&(ii)~ L^1(D) \ni x \mapsto \phi(x,\omega) ~\textnormal{is continuous for almost every}~\omega \in \Omega.
\end{align*}
Since it is possible to choose the filtration $\tilde{\mathcal{F}}_t$ instead of the filtration $(\mathcal{F}_t)_{t \in [s,T]}$, Theorem \ref{Theorem 6.2} yields that for fixed $x \in L^1(D)$ the function $u(t,s,x)$ is $\tilde{\mathcal{F}}_t$-measurable. Moreover, Theorem \ref{Theorem 7.1} yields that the mapping in (ii) is a contraction for almost every $\omega \in \Omega$, especially it is continuous.\\
Now, Lemma 4.1. in \cite{PB} is applicable and yields the assertion.
\end{proof}
For $s,t \in [0,T]$, $s \leq t$ and $x \in L^1(D)$ we set $P_{s,t}: B_b(L^1(D)) \to B_b(L^1(D))$,
\begin{align*}
P_{s,t}(\varphi)(x) = \mathbb{E} [\varphi(u(t,s,x))],
\end{align*}
where $B_b(L^1(D))$ denotes the space of all bounded Borel functions from $L^1(D)$ to $\mathbb{R}$. Moreover, we set $P_t:= P_{0,t}$.\\
\newline
As a consequence of Theorem \ref{Theorem 8.2} we obtain the Chapman-Kolmogorov property:
\begin{corollary}\label{Corollary 8.3}
For $r,s,t \in [0,T]$, $r \leq s \leq t$, $x \in L^1(D)$ and $\varphi \in B_b(L^1(D))$ we have
\begin{align}
P_{r,t}(\varphi)(x)= P_{r,s}(P_{s,t}(\varphi))(x).
\end{align}
\end{corollary}
\begin{proof}
Let $r,s,t \in [0,T]$, $r \leq s \leq t$, $x \in L^1(D)$ and $\varphi \in B_b(L^1(D))$. From Theorem \ref{Theorem 8.2} it follows that
\begin{align*}
P_{r,t}(\varphi)(x) &= \mathbb{E} [\varphi(u(t,r,x))] = \mathbb{E} [ \mathbb{E} [\varphi(u(t,r,x)) |\mathcal{F}_s ]] = \mathbb{E} [ \mathbb{E} [\varphi(u(t,s,u(s,r,x))) ] ] \\
&=\mathbb{E} [ P_{s,t}(\varphi)(u(s,r,x)) ] = P_{r,s}(P_{s,t}(\varphi))(x).
\end{align*}
\end{proof}
\begin{corollary}\label{Corollary 8.4}
For all $s,t \in [0,T]$, $s \leq t$, we get
\begin{align*}
P_{s,t}=P_{0,t-s}.
\end{align*}
In particular, $(P_t)_{t \in [0,T]}$ is a semigroup.
\end{corollary}
\begin{proof}
Similar as in Proposition \ref{Proposition 8.1} we can show that $u(\tau +s,s,x)=u^{\hat{\beta}}(\tau,0,x)$ for $s \in [0,T]$, $\tau \in [0,T-s]$, where $\hat{\beta}(\tau)= \beta(\tau+s) - \beta(s)$ and $u^{\hat{\beta}}(\tau,0,x)$ is the unique renormalized solution to \eqref{1} with respect to the Brownian motion $\hat{\beta}$ and initial value $x \in L^1(D)$. Since renormalized solutions to \eqref{1} are pathwise unique, they are jointly unique in law (see, e.g., \cite{O}, Theorem 2) and therefore we have
\begin{align*}
P_{s,s+\tau}= P_{0,\tau}.
\end{align*}
Setting $t= \tau + s$ yields the assertion.
\end{proof}
Now we show that $P_{s,t}$ is Feller and $(P_t)_{t \in [0,T]}$ is a Feller semigroup (see, e.g., \cite{GDPJZ}, p. 247):
\begin{proposition}
For all $s,t \in [0,T]$, $s \leq t$ we have $P_{s,t} (\mathcal{C}_b(L^1(D))) \subset \mathcal{C}_b(L^1(D))$.
\end{proposition}
\begin{proof}
Let $s,t \in [0,T]$, $s \leq t$ and $\varphi \in \mathcal{C}_b(L^1(D))$. Let $(x_n) \subset L^1(D)$ such that $x_n \to x$ in $L^1(D)$. Theorem \ref{Theorem 7.1} and the continuity of $\varphi$ yields
\begin{align*}
\varphi (u(t,s,x_n)) \to \varphi(u(t,s,x))
\end{align*}
a.e. in $\Omega$. Since $\varphi$ is bounded this convergence is also a convergence in $L^1(\Omega)$ by Lebesgue's Theorem. Therefore we have
\begin{align*}
P_{s,t}(\varphi)(x_n) = \mathbb{E} (\varphi (u(t,s,x_n))) \to \mathbb{E} (\varphi (u(t,s,x))) = P_{s,t}(\varphi)(x).
\end{align*}
Since $|P_{s,t}(\varphi)(x)| \leq \Vert \varphi \Vert_{\infty} < \infty$ for all $x \in L^1(D)$ we may conclude $P_{s,t}(\varphi) \in \mathcal{C}_b(L^1(D))$.
\end{proof}
\begin{proposition}
The family $P_{s,t}$, $s,t \in [0,T]$, $s \leq t$, has the e-property in the sense of \cite{TKSPTZ}, i.e.:\\
For all $\varphi \in \textnormal{Lip}_b(L^1(D))$, $x \in L^1(D)$ and $\epsilon>0$ there exists $\delta >0$ such that for all $z \in B(x, \delta)$ and all $0 \leq s \leq t \leq T$:
\begin{align*}
|P_{s,t}(\varphi)(x) - P_{s,t}(\varphi)(z)| < \epsilon,
\end{align*}
where $\textnormal{Lip}_b(L^1(D))$ denotes the space of all bounded Lipschitz continuous functions from $L^1(D)$ to $\mathbb{R}$.
\end{proposition}
\begin{proof}
Let $\varphi \in \textnormal{Lip}_b(L^1(D))$, $x \in L^1(D)$ and $\epsilon>0$ and let $L>0$ be a Lipschitz constant of  $\varphi$. We set $\delta:= \frac{\epsilon}{L}$. Then, for all $z \in B(x, \delta)$ and all $0 \leq s \leq t \leq T$ Theorem \ref{Theorem 7.1} yields
\begin{align*}
|P_{s,t}(\varphi)(x) - P_{s,t}(\varphi)(z)| &= | \mathbb{E} (\varphi(u(t,s,x)) - \varphi(u(t,s,z)))| \\
&\leq L \cdot \mathbb{E} \Vert u(t,s,x) - u(t,s,z) \Vert_1 \\
&\leq L \cdot \Vert x - z \Vert_1 < L \cdot \delta = \epsilon.
\end{align*}
\end{proof}
As in \cite{WLMR} we define for $x \in L^1(D)$
\begin{align*}
\mathbb{P}_x:= P \circ (u(\cdot, 0, x))^{-1},
\end{align*}
i.e., $\mathbb{P}_x$ is the distribution of the unique renormalized solution to \eqref{1} with initial condition $x \in L^1(D)$, defined as a probability measure on $\mathcal{C}([0,T]; L^1(D))$. We equip $\mathcal{C}([0,T]; L^1(D))$ with the $\sigma$-Algebra 
\begin{align*}
\mathcal{G}:= \sigma(\pi_s, ~s \in [0,T])
\end{align*}
and filtration
\begin{align*}
\mathcal{G}_t:= \sigma(\pi_s, ~s \in [0,t]), ~t \in [0,T],
\end{align*}
where $\pi_t: \mathcal{C}([0,T]; L^1(D)) \to L^1(D)$, $\pi_t(w):= w (t)$. Finally we can prove the following property of $\mathbb{P}_x$:
\begin{proposition}
$\mathbb{P}_x$, $x \in L^1(D)$, is a time-homogenous Markov process on $\mathcal{C}([0,T]; L^1(D))$ with respect to the filtration $(\mathcal{G}_t)_{t \in [0,T]}$, i.e., for all $s,t \in [0,T]$ such that $s+t \leq T$ and all $\varphi \in B_b(L^1(D))$ we have
\begin{align}\label{46}
\mathbb{E}_x (\varphi(\pi_{t+s})|\mathcal{G}_s) = \mathbb{E}_{\pi_s} (\varphi(\pi_t))
\end{align}
$\mathbb{P}_x$-a.s., where $\mathbb{E}_x$ and $\mathbb{E}_x( \cdot|\mathcal{G}_s)$ denote the expectation and the conditional expectation with respect to $\mathbb{P}_x$, respectively.
\end{proposition}
\begin{proof}
We start the proof by showing the right-hand side of \eqref{46} to be $\mathcal{G}_s$-measurable. This will be done by applying a so-called monotone class argument. To this end we set
\begin{align*}
\mathcal{H}:= \{ \varphi: L^1(D) \to \mathbb{R}, ~\mathbb{E}_{\pi_s}(\varphi(\pi_t)): \mathcal{C}([0,T]; L^1(D)) \to \mathbb{R} ~\textnormal{is}~ \mathcal{G}_s-\textnormal{measurable} \}
\end{align*}
and we show that $B_b(L^1(D)) \subset \mathcal{H}$. Firstly we mention and prove that $\mathcal{H}$ satisfies the following properties:
\begin{align*}
i&)~ \textnormal{If}~ A \in \mathcal{B}(L^1(D)), ~\textnormal{then}~ \chi_A \in \mathcal{H}, \\
ii&)~ \textnormal{If}~ f,g \in \mathcal{H} ~\textnormal{and}~c\in \mathbb{R}, ~\textnormal{then}~ f+g \in \mathcal{H} ~\textnormal{and}~cf \in \mathcal{H}, \\
iii&) ~\textnormal{If} ~f_n \in \mathcal{H}, ~0 \leq f_n \nearrow f ~\textnormal{and}~ f ~\textnormal{bounded, then} ~f \in \mathcal{H}.
\end{align*}
To i): Let $A \in \mathcal{B}(L^1(D))$. Then we have for arbitrary $w \in C([0,T];L^1(D))$:
\begin{align*}
\mathbb{E}_{\pi_s} (\chi_A(\pi_t))(w) &= \mathbb{E}_{w(s)} (\chi_A(\pi_t)) = \int \chi_A(\pi_t) \, dP_{w(s)} = \mathbb{E} [\chi_A(u(t,0,w(s)))] \\ 
&= P_t(\chi_A)(w(s)).
\end{align*}
Now let $B \in \mathcal{B}(\mathbb{R})$ and set $\tilde{B} := [P_t(\chi_A)]^{-1} (B) \in \mathcal{B}(L^1(D))$. Then
\begin{align*}
[\mathbb{E}_{\pi_s} (\chi_A(\pi_t))]^{-1}(B) &= \{w \in \mathcal{C}([0,T];L^1(D)), ~P_t(\chi_A)(w(s)) \in B\} \\
&= \{w \in \mathcal{C}([0,T];L^1(D)), ~w(s) \in \tilde{B}\} \\
&= (\pi_s)^{-1}(\tilde{B}) \in \mathcal{G}_s.
\end{align*}
To ii): This is obvious since the sum and the product of real valued measurable functions is again measurable.\\
To iii): Let $f_n \in \mathcal{H}$ and $0 \leq f_n \nearrow f$, where $f$ is a bounded function. Then for arbitrary $w \in \mathcal{C}([0,T];L^1(D))$ we have
\begin{align*}
\mathbb{E}_{\pi_s} (f_n(\pi_t))(w) = \mathbb{E} [f_n(u(t,0,w(s)))] \to \mathbb{E} [f(u(t,0,w(s)))] = \mathbb{E}_{\pi_s} (f(\pi_t))(w)
\end{align*}
by the monotone convergence theorem. As a pointwise limit of $\mathcal{G}_s$-measurable functions it follows that $\mathbb{E}_{\pi_s} (f(\pi_t))$ is $\mathcal{G}_s$-measurable.\\
Now, properties i) and ii) yield that $\mathcal{H}$ contains all simple and Borel measurable functions and property iii) yields that $\mathcal{H}$ contains all bounded and Borel measurable functions, i.e., we may conclude $B_b(L^1(D)) \subset \mathcal{H}$. This means that for all $\varphi \in B_b(L^1(D))$ the function $\mathbb{E}_{\pi_s} (\varphi(\pi_t))$ is $\mathcal{G}_s$-measurable.\\
\newline
For the rest of the proof we follow the ideas in \cite{GDPJZ}. Let, for arbitrary $n \in \mathbb{N}$, $G: L^1(D)^n \to \mathbb{R}$ be a bounded $\otimes_{i=1}^{n} \mathcal{B}(L^1(D))$-measurable function and $0 \leq t_1 < ... < t_n \leq s$. Then from Theorem \ref{Theorem 8.2} and Corollary \ref{Corollary 8.4} it follows
\begin{align*}
&\mathbb{E}_x [ G(\pi_{t_1},..., \pi_{t_n}) \varphi(\pi_{t+s})] \\
= &\mathbb{E}[ G(u(t_1,0,x),..., u(t_n,0,x)) \varphi(u(t+s,0,x))] \\
= &\mathbb{E}\bigg[ \mathbb{E}[ G(u(t_1,0,x),..., u(t_n,0,x)) \varphi(u(t+s,0,x)) | \mathcal{F}_s ]\bigg]\\
= &\mathbb{E}\bigg[ G(u(t_1,0,x),..., u(t_n,0,x)) \mathbb{E} [ \varphi(u(t+s,0,x)) |\mathcal{F}_s ]\bigg] \\
= &\mathbb{E}\bigg[ G(u(t_1,0,x),..., u(t_n,0,x)) \mathbb{E} [ \varphi(u(t+s,s,u(s,0,x)))]\bigg] \\
= &\mathbb{E}\bigg[ G(u(t_1,0,x),..., u(t_n,0,x)) \mathbb{E} [ \varphi(u(t,0,u(s,0,x)))]\bigg]\\
= &\mathbb{E}_x\bigg[ G(\pi_{t_1},..., \pi_{t_n}) \mathbb{E} [ \varphi(u(t,0,\pi_s))]\bigg] \\
= &\mathbb{E}_x\bigg[ G(\pi_{t_1},..., \pi_{t_n}) \mathbb{E}_{\pi_s} [ \varphi(\pi_t))]\bigg].
\end{align*}
This yields the assertion.
\end{proof}

\section{Appendix}
\subsection{Proof of Lemma \ref{190212_lem01}}\label{proof}
For $n\in\mathbb{N}$, we define the following disjoint subdivision of $D$:
\[D_n:=\{x\in D \ | \ \operatorname{dist}(x,\partial D)\geq \frac{2}{n}\},\]
\[B_n:=\{x\in D \ | \ \operatorname{dist}(x,\partial D)\leq \frac{1}{n}\}\]
\[H_n:=\{x\in D \ | \ \frac{1}{n}<\operatorname{dist}(x,\partial D)<\frac{2}{n}\}.\]
In particular, $(D_n)_{n\in\mathbb{N}}$ is an increasing sequence of domains in $D$ such that $D_n\subset\subset D_{n+1}\subset D$ for all $n\in\mathbb{N}$ with
\[\bigcup_{n\in\mathbb{N}}D_n=D.\]
We choose a sequence of cutoff functions $(\varphi_n)_{n\in\mathbb{N}}:D\rightarrow\mathbb{R}$ such that $\varphi_n\in C^{\infty}_c(D)$, $0\leq \varphi_n\leq 1$ in $D$, $\varphi_n\equiv 1$ on $D_n$, $\varphi_n\equiv 0$ on $B_n$ and $|\nabla\varphi_n|\leq 2n$  for all $n\in\mathbb{N}$. Let $(\rho_n)_{n\in\mathbb{N}}\subset C^{\infty}_c(\mathbb{R}^d)$ be a sequence of symmetric mollifiers with support in $[-\frac{1}{n},\frac{1}{n}]$. For $n\in\mathbb{N}$ we define the linear operator 
\[\Pi_n:W^{-1,p'}(D)+L^1(D)\rightarrow W^{1,p}_0(D)\cap L^{\infty}(D),\]
\[v\mapsto (\varphi_nv)\ast\rho_n.\] 
We recall that $v\in W^{-1,p'}(D)+L^1(D)$ iff there exist $G\in L^{p'}(D)^d$, $f\in L^1(D)$ such that $v=-\operatorname{div}\,G+f$ in $\mathcal{D}'(D)$ and, according to the multiplication and convolution of distributions (see, e.g., \cite{Wloka}, Def. 1.5., p. 15 and Def. 1.6., p. 20)
\begin{align}\label{190319_01}
\Pi_n(v)(x)&=((\varphi_n f)\ast\rho_n)(x)+\langle -\operatorname{div}\,G, \varphi_n(\cdot)\rho_n(x-\cdot)\rangle_{W^{-1,p'}(D),W^{1,p}_0(D)}\nonumber\\
&=\int_{D}\rho_n(x-y)\varphi_n(y)f(y)\, dy+\int_DG(y)\nabla_y[\rho_n(x-y)\varphi_n(y)]\,dy
\end{align}
for all $x\in\mathbb{R}^d$. From the definition of $\Pi_n$ it follows immediately that $\Pi_n$ is linear and from \eqref{190319_01} we get that $\Pi_n(v)$ is a smooth function with $\Pi_n(v)=0$ on $D^C$ for all $n\in\mathbb{N}$. A straightforward calculation shows that, for arbitrary $v=-\operatorname{div}\, G+f\in W^{-1,p'}(D)+L^1(D)$, there exists a constant $C\geq 0$ not depending on $f$ and $G$ that may depend on $n\in\mathbb{N}$, such that  
\begin{align}\label{190321_01}
&\Vert \Pi_n v\Vert_{W^{1,p}_0(D)\cap L^{\infty}(D)}=\nonumber\\
&\max(\Vert \Pi_n(v)\Vert_{L^{\infty}(D)},\Vert \Pi_n(v)\Vert_{W^{1,p}_0(D)})\leq C (\Vert f\Vert_{L^1(D)}+\Vert -\operatorname{div}\, G\Vert_{W^{-1,p'}(D)})
\end{align} 
and, passing to the infimum over all $f\in L^1(D)$, $G\in L^{p'}(D)^d$ such that $v=f-\operatorname{div}\, G$ in \eqref{190321_01}, we get that $\Pi_n$ is a bounded linear operator from $W^{-1,p'}(D)+L^1(D)$ into $W^{1,p}_0(D)\cap L^{\infty}(D)$ for any $n\in\mathbb{N}$. For $F\in \{W^{1,p}_0(D), L^2(D), L^1(D)\}$ and every $v\in F\subset W^{-1,p'}(D)+L^1(D)$, from the classical properties of the convolution and Young inequality it follows that $\Pi_n\in L(F)$ for any $n\in\mathbb{N}$ and $\Pi_n(v)\rightarrow v$ for $n\rightarrow\infty$ in $F$ for $n\rightarrow\infty$. For arbitrary $v\in W^{-1,p'}(D)+L^1(D)$, 
\[\lim_{n\rightarrow\infty}\Vert \Pi_n(v)-v\Vert_{W^{-1,p'}(D)+L^1(D)}=0\]
iff
\begin{align}\label{190322_06}
\lim_{n\rightarrow\infty}\left(\Vert \Pi_n(f)-f\Vert_{L^1(D)}+\Vert\Pi_n(-\operatorname{div}\, G)-(-\operatorname{div}\, G)\Vert_{W^{-1,p'}(D)}\right)=0
\end{align}
for all $f\in L^1(D)$, $G\in L^{p'}(D)^d$ such that $v=f-\operatorname{div}\, G$. Thus, to conclude the proof, the convergence of $\Pi_n(-\operatorname{div}\, G)$ to $-\operatorname{div}\, G$ for $n\rightarrow\infty$ in $W^{-1,p'}(D)$ for arbitrary $G\in L^{p'}(D)^d$ deserves our attention. For $g\in W^{1,p}_0(D)$, we have
\begin{align*}
&\left|\langle \Pi_n(-\operatorname{div}\, G),g\rangle_{W^{-1,p'}(D),W^{1,p}_0(D)}\right|\\
&=\left|\int_D g(x)\langle-\operatorname{div}\, G,\varphi_n(\cdot)\rho_n(x-\cdot)\rangle_{W^{-1,p'}(D),W^{1,p}_0(D)}\,dx\right|\\
&=\left|\int_D\int_D G(y)\cdot[(\nabla_y\varphi_n(y))\rho_n(x-y)+(\nabla_y\rho_n(x-y))\varphi_n(y)]\,dy \,g(x)\,dx\right|\\
&\leq I_1^n+I_2^n,
\end{align*}
where 
\begin{align*}
I_1^n&=\left|\int_D\int_D G(y)\cdot(\nabla_y\varphi_n(y))\rho_n(x-y)\,dy\,g(x)\,dx\right|,\\
I_2^n&=\left|\int_D\int_D G(y)\cdot(\nabla_y\rho_n(x-y))\varphi_n(y)\,dy \,g(x)\,dx\right|.
\end{align*}
Recalling that $\nabla_y\rho_n(x-y)=-\nabla_x\rho_n(x-y)$ using Fubini's theorem and Young's inequality it follows that
\begin{align}\label{190321_02}
I_2^n&=\left|\int_D\varphi_n(y)G(y)\cdot\int_D\nabla_x\rho_n(x-y)g(x)\,dx\,dy\right|\nonumber\\
&=\left|\int_D\varphi_n(y)G(y)\cdot\nabla_y[\rho_n\ast g](y)\,dy\right|\nonumber\\
&\leq \Vert G\Vert_{L^{p'}(D)^d}\Vert\nabla g\Vert_{L^p(D)}
\end{align}
for all $n\in\mathbb{N}$. Thanks to Fubini's theorem and to the properties of $\nabla\varphi_n$ and using H\"older and Young's inequality we get
\begin{align}
I_1^n&=\left|\int_{H_n}\int_D\rho_n(x-y)g(x)\,dx\,G(y)\cdot(\nabla_y\varphi_n(y))\,dy\right|\nonumber\\
&\leq \int_{H_n}|n(\rho_n\ast g)(y)||G(y)|\,dy\nonumber\\
&\leq \Vert G\Vert_{L^{p'}(H_n)^d}\Vert n(\rho_n\ast g)\Vert_{L^p(H_n)}\nonumber\\
&\leq\Vert G\Vert_{L^{p'}(H_n)^d}\left[\int_{H_n}\left(\frac{|g(y)|}{\frac{1}{n}}\right)^p\,dy\right]^{1/p}.
\end{align}
Recalling that for all $y\in H_n$ we have $\operatorname{dist}(y,\partial D)<\frac{2}{n}$ it follows that 
\begin{align*}
I_1^n\leq 2^p\Vert G\Vert_{L^{p'}(H_n)^d}\left[\int_{D}\left(\frac{|g(y)|}{\operatorname{dist}(y,\partial D)}\right)^p\,dy\right]^{1/p}.
\end{align*}
Now, using Hardy's inequality we conclude that there exists a constant $C\geq 0$ not depending on $n\in\mathbb{N}$ such that
\begin{align}\label{190322_01}
I_1^n\leq C\Vert G\Vert_{L^{p'}(H_n)^d}\Vert \nabla g\Vert_{L^p(D)^d}.
\end{align}
From \eqref{190321_02} and \eqref{190322_01} it follows that 
\begin{align}\label{190322_03}
\Vert \Pi_n(-\operatorname{div}\,G)\Vert_{W^{-1,p'}(D)}\leq \Vert G\Vert_{L^{p'}(D)^d}+C\left(\int_{H_n}|G(y)|^{p'}\,dy\right)^{1/p'}
\end{align}
for all $n\in\mathbb{N}$, and therefore $\Vert \Pi_n(-\operatorname{div}\,G)\Vert_{W^{-1,p'}(D)}$ is bounded with respect to $n\in\mathbb{N}$ for any $v=-\operatorname{div}\,G\in W^{-1,p'}(D)$. The proof of Theorem 4.15 in \cite{LH} yields that
\begin{align*}
\Pi_n(-\operatorname{div}\,G) \to - \operatorname{div}\,G 
\end{align*}
in $\mathcal{D}'(D) = (C_0^{\infty}(D))^*$.
Hence by density of $C_0^{\infty}(D)$ in $W_0^{1,p}(D)$ and boundedness of $\Pi_n(-\operatorname{div}\,G)$ in $W^{-1,p'}(D)$ we get
\begin{align}\label{190322_04}
\lim_{n\rightarrow\infty}\Pi_n(-\operatorname{div}\,G)=-\operatorname{div}\,G
\end{align}
weakly in $W^{-1,p'}(D)$. Finally, we remark that from \eqref{190322_03} we also get
\begin{align}\label{190322_05}
\limsup_{n\rightarrow\infty}\Vert \Pi_n(-\operatorname{div}\,G)\Vert_{W^{-1,p'}(D)}\leq\Vert -\operatorname{div}\,G\Vert_{W^{-1,p'}(D)}.
\end{align}
Now, from \eqref{190322_05} and the uniform convexity of $W^{-1,p'}(D)$ it follows that \eqref{190322_04} holds strongly in $W^{-1,p'}(D)$ and therefore \eqref{190322_06} holds true. In particular, we have obtained $\Pi_n\in L(F)$ and $\Pi_n(v)\rightarrow v$ for $v\in F$ and $n\rightarrow\infty$ in the case $F=W^{-1,p'}(D)$ and $F=W^{-1,p'}(D)+L^1(D)$.

\subsection{The It\^{o} product rule}
In the well-posedness theory of renormalized solutions in the deterministic setting (see, e.g., \cite{DB}), the product rule is a crucial part. In the following lemma, we propose an It\^{o} product rule for strong solutions to \eqref{1}. In the following, we will call a function $f:\mathbb{R}\rightarrow\mathbb{R}$ \textit{piecewise continuous}, iff it is continuous except for finitely many points.

\begin{prop}\label{itoproduct}
For $1<p<\infty$, $u_0$, $v_0\in L^2(\Omega\times D)$ $\mathcal{F}_0$-measurable let $u$ be a strong solution to \eqref{1} with initial datum $u_0$ and $v$ be a strong solution to \eqref{1} with initial datum $v_0$ respectively. 
Then, for any $H\in \mathcal{C}^2_b(\mathbb{R})$ and any $Z\in W^{2,\infty}(\mathbb{R})$ with $Z''$ piecewise continuous such that $Z(0)=Z'(0)=0$
\begin{align}\label{181201_04}
&(Z((u-v)(t)),H(u(t)))_2=(Z(u_0-v_0),H(u_0))_2\nonumber\\
&+\int_0^t\langle \Delta_p(u)-\Delta_p(v),H(u)Z'(u-v)\rangle_{W^{-1,p'}(D),W^{1,p}_0(D)}\, ds\nonumber\\
&+\int_0^t\langle \Delta_p(u),H'(u)Z(u-v)\rangle_{W^{-1,p'}(D),W_0^{1,p}(D)}\, ds+\int_0^t(\Phi H'(u),Z(u-v))_2\,d\beta\nonumber\\
&+\frac{1}{2}\int_0^t\int_D\Phi^2H''(u)Z(u-v)\, dx\,ds
\end{align}
for all $t\in [0,T]$ a.s. in $\Omega$.
\end{prop}
\begin{proof}
We fix $t\in [0,T]$. Since $u$, $v$ are strong solutions to \eqref{1}, it follows that
\begin{align}\label{181201_03}
u(t)=u_0+\int_0^t\Delta_p(u)\,ds+\int_0^t\Phi\,d\beta,
\end{align}
\begin{align*}
v(t)=v_0+\int_0^t\Delta_p(v)\,ds+\int_0^t\Phi\,d\beta
\end{align*}
and consequently
\begin{align}\label{181201_02}
(u-v)(t)=u_0-v_0+\int_0^t \Delta_p(u)-\Delta_p(v)\,ds
\end{align}
holds in $L^2(D)$, a.s. in $\Omega$. 
For $n\in\mathbb{N}$ we define $\Pi_n$ according to Lemma \ref{190212_lem01} and set $\Phi_n:=\Pi_n(\Phi)$, $u^{n}_0:=\Pi_n(u_0)$,  $v^{n}_0:=\Pi_n(v_0)$, $u_{n}:=\Pi_n(u)$, $v_{n}:=\Pi_n(v)$, $U_n:=\Pi_n(\Delta_p(u))$, $V_n:=\Pi_n(\Delta_p(v))$. Applying $\Pi_n$ on both sides of \eqref{181201_02} yields
\begin{align}\label{181203_01}
(u_{n}-v_{n})(t)=u^{n}_0-v^{n}_0+\int_0^t U_n-h_{n}\,ds
\end{align}
and applying $\Pi_n$ on both sides of \eqref{181201_03} yields
\begin{align}\label{181203_02}
u_{n}(t)=u_0^{n}+\int_0^t g_{n}\,ds+\int_0^t \Phi_n \,d\beta\nonumber\\
\end{align}
in $W^{1,p}_0(D)\cap L^2(D)\cap \mathcal{C}^{\infty}(\overline{D})$ a.s. in $\Omega$.
The pointwise It\^{o} formula in \eqref{181203_01} and \eqref{181203_02} leads to
\begin{align}\label{181212_01}
Z(u_{n}-v_n)(t)=Z(u^{n}_0-v_0^{n})+\int_0^t (U_n-h_{n})Z'(u_n-v_n)\,ds
\end{align}
and
\begin{align}\label{181212_02}
H(u_{n})(t)=H(u^{n}_0)+\int_0^t g_{n}H'(u_{n})\,ds
+\int_0^t\Phi_n H'(u_n)\,d\beta+\frac{1}{2}\int_0^t\Phi_n^2H''(u_{n})\, ds
\end{align}
in $D$, a.s. in $\Omega$. From \eqref{181212_01}, \eqref{181212_02} and the  product rule for It\^{o} processes, which is just an easy application of the classic two-dimensional It\^{o} formula (see, e.g., \cite{PB}, Proposition 8.1, p. 218), applied pointwise in $t$ for fixed $x\in D$ it follows that
 \begin{align}\label{181212_03}
&Z(u_{n}-v_{n})(t)H(u_{n})(t)=Z(u^{n}_0-v^{n}_0)H(u^{n}_0)\nonumber\\
 &+\int_0^t (g_{n}-V_n)Z'(u_n-v_n)H(u_n)\,ds+\int_0^t U_n H'(u_n)Z(u_n-v_n)\,ds\nonumber\\
 &+\int_0^t \Phi_n H'(u_n)Z(u_n-v_n)\,d\beta+\frac{1}{2}\int_0^t\Phi_n^2H''(u_n)Z(u_n-v_n)\,ds
\end{align}
in $D$, a.s. in $\Omega$. Integration  over $D$ in \eqref{181212_03} yields
\begin{align}\label{181212_04}
I_1=I_2+I_3+I_4+I_5+I_6
\end{align}
where
\begin{align*}
I_1&=(Z((u_n-v_n)(t)),H((u_n)(t))_2, \nonumber\\
I_2&=(Z(u^{n}_0-v^{n}_0),H(u^n_0))_2, \nonumber\\
I_3&=\int_0^t\int_D(g_{n}-V_n)Z'(u_n-v_n)H(u_n)\,dx\,ds, \nonumber\\
I_4&=\int_0^t\int_D U_n H'(u_n)Z(u_n-v_n)\,dx\,ds, \nonumber\\
I_5&=\int_0^t (\Phi_n H'(u_n),Z(u_n-v_n))_2\,d\beta, \nonumber\\
I_6&=\frac{1}{2}\int_0^t\int_D\Phi_n^2 H''(u_n)Z(u_n-v_n)\,dx\,ds
\end{align*}
a.s. in $\Omega$. For any fixed $s\in [0,t]$ and almost every $\omega\in \Omega$, $u_n(\omega,s)\rightarrow u(\omega,s)$ and $v_n(\omega,s)\rightarrow v(\omega,s)$ for $n\rightarrow\infty$ in $L^2(D)$. Since $Z$, $H$, $H'$ are continuous and bounded functions, it follows that
\begin{align}\label{181212_05}
\lim_{n\rightarrow\infty}I_1=(Z((u-v)(t)),H'(u(t))_2,
\end{align}
\begin{align}\label{181212_06}
\lim_{n\rightarrow\infty}I_2=(Z(u_0-v_0),H'(u_0))_2
\end{align}
in $L^2(\Omega)$ and a.s. in $\Omega$. 
Note that
\begin{align*}
I_3=\int_0^t\langle(U_n-V_n),Z'(u_n-v_n)H(u_n)\rangle_{W^{-1,p'}(D),W_0^{1,p}(D)}\,ds
\end{align*}
a.s. in $\Omega$ and from the properties of $\Pi_n$ it follows that
\[\lim_{n\rightarrow\infty}U_n(\omega,s)-V_n(\omega,s)=\Delta_p(u(\omega,s))-\Delta_p(v(\omega,s))\]
in $W^{-1,p'}(D)$ for all $s\in [0,t]$ and a.e. $\omega\in\Omega$. Recalling the convergence result for $(\Pi_n)$ from Lemma \ref{190212_lem01}, there exists a constant $C_1\geq 0$ not depending on $s,\omega$ and $n\in\mathbb{N}$ such that
\begin{align*}
\Vert U_n(\omega,s)-V_n(\omega,s)\Vert_{W^{-1,p'}(D)}&=\Vert \Pi_n(\Delta_p(u(\omega,s))-\Delta_p(v(\omega,s)))\Vert_{W^{-1,p'}(D)}\\
&\leq C_1 \Vert \Delta_p(u(\omega,s))-\Delta_p(v(\omega,s))\Vert_{W^{-1,p'}(D)}.
\end{align*}
Since the right-hand side of the above equation is in $L^{p'}(\Omega\times (0,t))$, from Lebesgue's dominated convergence theorem it follows that
\[\lim_{n\rightarrow\infty} U_n-V_n=\Delta_p(u)-\Delta_p(v)\]
in $L^{p'}(\Omega\times (0,t);W^{-1,p'}(D))$ and, with a similar reasoning, also in $L^{p'}(0,t;W^{-1,p'}(D))$ a.s. in $\Omega$. From the chain rule for Sobolev functions it follows that
\begin{align}\label{190213_01}
\nabla(Z'(u_n-v_n)H(u_n))=Z''(u_n-v_n)\nabla(u_n-v_n)H(u_n)+Z'(u_n-v_n)H'(u_n)\nabla u_n
\end{align}
a.s. in $(0,t)\times \Omega$. Moreover, there exists a constant $C_2=C_2(\Vert Z'\Vert_{\infty},\Vert Z''\Vert_{\infty},\Vert H\Vert_{\infty},\Vert H'\Vert_{\infty})\geq 0$ such that
\begin{align}\label{181218_01}
\int_0^t \Vert \nabla(Z'(u_n-v_n)H(u_n))\Vert_p^p \, ds\leq C_2\int_0^t(\Vert\nabla u\Vert^p_p+\Vert\nabla v\Vert^p_p)\, ds
\end{align}
a.s. in $\Omega$. Consequently, for almost every $\omega\in\Omega$ there exists $\chi(\omega)\in L^p(0,t;W^{1,p}_0(D))$ such that, passing to a not relabeled subsequence that may depend on $\omega\in\Omega$, 
\begin{align}\label{190212_02}
Z'(u_n-v_n)H(u_n)\rightharpoonup\chi(\omega)
\end{align}
weakly in $L^p(0,t; W^{1,p}_0(D))$. Since in addition, 
\[\lim_{n\rightarrow\infty}Z'(u_n-v_n)H(u_n)\rightarrow Z'(u-v)H(u)\]
in $L^p((0,t)\times D)$ a.s. in $\Omega$, we get
\begin{align}\label{190212_03}
\chi(\omega)=Z'(u-v)H(u)
\end{align}
in $L^p(0,t;W^{1,p}_0(D))$ a.s. in $\Omega$ and the weak convergence in \eqref{190212_02} holds for the whole sequence. Therefore,
\[Z'(u_n-v_n)H(u_n)\rightharpoonup Z'(u-v)H(u)\]
for $n\rightarrow\infty$ weakly in $L^p(0,t;W^{1,p}_0(D))$ for almost every $\omega\in\Omega$. Resuming the above results it follows that
\begin{align}\label{181218_03}
\lim_{n\rightarrow\infty}I_3=\int_0^t \langle \Delta_p(u)-\Delta_p(v) ,Z'(u-v)H(u)\rangle_{W^{-1,p'}(D),W_0^{1,p}(D)}\,ds
\end{align}
a.s. in $\Omega$. With analogous arguments we get
\begin{align}\label{181212_12}
\lim_{n\rightarrow\infty}I_4=\int_0^t\langle \Delta_p(u),H'(u)Z(u-v)\rangle_{W^{-1,p'}(D),W_0^{1,p}(D)}\, ds
\end{align}
a.s. in $\Omega$. By It\^{o} isometry,
\begin{align*}
&\mathbb{E}\left|\int_0^t\int_D\Phi_n H'(u_n)Z(u_n-v_n)-\Phi H'(u)Z(u-v)\,dx\,d\beta\right|^2\\
&=\mathbb{E}\int_0^t\int_D |\Phi_n H'(u_n)Z(u_n-v_n)-\Phi H'(u)Z(u-v)|^2\,dx\,ds.
\end{align*}
From the convergence
\[\Phi_n H'(u_n)Z(u_n-v_n)\rightarrow
\Phi H'(u)Z(u-v)\]
in $L^2(D)$ for $n\rightarrow\infty$ a.s. in $\Omega\times (0,t)$ and since, for almost any $(\omega,s)$, there exists a constant $C_3\geq 0$ not depending on the parameters $n,s,\omega$ such that
\[\Vert \Phi_n(\omega,s) H'(u_n(\omega,s))Z(u_n(\omega,s)-v_n(\omega,s))\Vert_2\leq C_3\Vert \Phi(\omega,s)\Vert_2\]
for all $n\in\mathbb{N}$, a.s. in $\Omega\times (0,t)$, it follows that
\[\lim_{n\rightarrow\infty}\Phi_n H'(u_n)Z(u_n-v_n)=
\Phi H'(u)Z(u-v)\]
in $L^2(\Omega\times(0,t)\times D)$ and consequently  
\begin{align}\label{181212_13}
\lim_{n\rightarrow\infty} I_5=\int_0^t\int_D\Phi H'(u)Z(u-v)\,dx\,d\beta
\end{align}
in $L^2(\Omega)$ and, passing to a subsequence if necessary, also a.s. in $\Omega$. According to the properties of $(\Pi_n)$, $\Phi_n^2\rightarrow\Phi^2$ in $L^1((0,t)\times D)$ for $n\rightarrow\infty$ a.s. in $\Omega$. From the boundedness and the continuity of $H''$ and $Z$ we get
\[\lim_{n\rightarrow\infty} H''(u_n)Z(u_n-v_n)=H''(u)Z(u-v)\]
in $L^{q}((0,t)\times D)$ for all $1\leq q<\infty$ and weak-$\ast$ in  $L^{\infty}((0,t)\times D)$ a.s. in $\Omega$, thus it follows that
\begin{align}\label{181213_01}
\lim_{n\rightarrow\infty}I_6=\frac{1}{2}\int_0^t\int_D\Phi^2 H''(u)Z(u-v)\,dx\,ds
\end{align}
a.s. in $\Omega$. Passing to a subsequence if necessary, taking the limit in \eqref{181212_03} for $n\rightarrow\infty$ a.s. in $\Omega$ 
the assertion follows from \eqref{181212_05}-\eqref{181213_01}.
\end{proof}
\begin{cor}\label{cor1}
Proposition \ref{itoproduct} still holds true for $H\in W^{2,\infty}(\mathbb{R})$ such that $H''$ is piecewise continuous.
\end{cor}
\begin{proof}
There exists an approximating sequence $(H_{\delta})_{\delta>0}\subset \mathcal{C}^2_b(\mathbb{R})$ such that $\Vert H_{\delta}\Vert_{\infty}\leq \Vert H\Vert_{\infty}$, $\Vert H'_{\delta}\Vert_{\infty}\leq \Vert H'\Vert_{\infty}$, $\Vert H''_{\delta}\Vert_{\infty}\leq \Vert H''\Vert_{\infty}$ for all $\delta>0$ and $H_{\delta}\rightarrow H$, $H'_{\delta}\rightarrow H'$ uniformly on compact subsets, $H_{\delta}''\rightarrow H''$ pointwise in $\mathbb{R}$ for $\delta\rightarrow 0$. With this convergence we are able to pass to the limit with $\delta\rightarrow 0$ in \eqref{181201_04}.
\end{proof}

\end{document}